\documentclass[preprint]{elsarticle}

\usepackage[english]{babel}
\usepackage{amsmath}
\usepackage{amssymb}
\usepackage{amsthm}
\usepackage{stackrel}
\usepackage{subfloat}
\usepackage{subfig}
\usepackage{setspace}
\usepackage{enumerate}
\usepackage{algorithmic}
\usepackage{nomencl}
\usepackage{easybmat}
\usepackage{amsmath}
\usepackage{graphicx}
\newtheorem{defnt}{Definition}

\newtheorem{lemma}{Lemma}
\newtheorem{remk}{Remark}

\newtheorem{thm}{Theorem}
\journal{}

\begin{document}
\begin{frontmatter}
\title{Lie Bracket Approximation of Extremum Seeking Systems}
\author[unist]{Hans-Bernd D\"urr}
\author[KTH]{Milo\v{s} S. Stankovi\'{c}}
\author[unist]{Christian Ebenbauer}
\author[KTH]{Karl Henrik Johansson}

\address[unist]{Institute for Systems Theory and Automatic Control, University of Stuttgart, Germany; (e-mail: {hans-bernd.duerr,ce}@ist.uni-stuttgart.de)}  
\address[KTH]{ACCESS Linnaeus Center, School of Electrical Engineering, KTH Royal Institute of Technology, 100 44 Stockholm, Sweden;
(e-mail: {milsta,kallej}@kth.se)}             

\begin{abstract}
Extremum seeking feedback is a powerful method to steer a dynamical system to an extremum of a partially or completely unknown map. It often requires advanced system-theoretic tools to understand the qualitative behavior of extremum seeking systems. In this paper, a novel interpretation of extremum seeking is introduced. We show that the trajectories of an extremum seeking system can be approximated by the trajectories of a system which involves certain Lie brackets of the vector fields of the extremum seeking system. It turns out that the Lie bracket system directly reveals the optimizing behavior of the extremum seeking system. Furthermore, we establish a theoretical foundation and prove that uniform asymptotic stability of the Lie bracket system implies practical uniform asymptotic stability of the corresponding extremum seeking system. We use the established results in order to prove local and semi-global practical uniform asymptotic stability of the extrema of a certain map for multi-agent extremum seeking systems.
\end{abstract}
\end{frontmatter}

\section{Introduction}
In diverse engineering applications one faces the problem of finding an extremum of a map without knowing its explicit analytic expression. 
Suppose, for example, one vehicle tries to minimize the distance to another vehicle. The only information available, is the distance to the other vehicle. Clearly, the distance does not provide a direction in which the vehicle has to move. However, it is intuitively clear that one can obtain a direction by using multiple measurements of the distance. 
Extremum seeking feedback exploits this procedure in a systematic way and can be used for steering dynamical systems to the extremum of an unknown map. Extremum seeking has a long history and has found many applications to diverse problems in control and communications (see \cite{5572972} and references therein). 


In this paper, we provide a novel methodology to analyze extremum seeking systems which differs from commonly used techniques (see e.g. \cite{Sanders:2007uq}).
Specifically, this work contains three main contributions. 

First, we provide a novel view on extremum seeking by identifying the sinusoidal perturbations in the extremum seeking system as artificial inputs and by writing it in a certain input-affine form. Based on this input-affine form, we derive an approximate system which captures the behavior of the trajectories of the original extremum seeking system. It turns out that the approximate system can be represented by certain Lie brackets of the vector fields in the extremum seeking system. We call this approximate system the Lie bracket system. The proposed methodology is different from results in the existing literature (see e.g. \cite{Krstic:2003fk}, \cite{Krstic:2000kx} and \cite{nesic}). 

Second, we establish a theoretic foundation which is based on this novel viewpoint. We prove that the trajectories of a class of input-affine systems with certain inputs are approximated by the trajectories of the Lie bracket systems. Similar results concerning sinusoidal inputs are covered in \cite{Kurzweil:1987fk} and were extended in \cite{220059}, \cite{573210} to the class of periodic inputs. In \cite{Sussmann:1991ys} and \cite{H.-J.-Sussmann:1993vn} convergence of trajectories of a class of input-affine systems to the trajectories of more general Lie bracket systems was established. These results are closely related to our results. Furthermore, we prove under mild assumptions that semi-global (local) practical uniform asymptotic stability of a class of input-affine systems follows from global (local) uniform asymptotic stability of the corresponding Lie bracket systems. These results are based on \cite{871771} and \cite{Moreau:2003fk}. Summarizing, to the authors best knowledge, the generality of the setup proposed herein was not addressed in the literature before.

Third, we apply the established results to analyze the behavior and the stability properties of extremum seeking vehicles with single-integrator and unicycle dynamics and with static maps. 
We formulate a multi-agent setup consisting of extremum seeking systems where the individual nonlinear maps of the agents satisfy a certain relationship which assures the existence of a potential function. We use the established theoretical results to show that the set of extrema of the potential function is (locally or semi-globally) practically uniformly asymptotically stable for the multi-agent system. This multi-agent setup is strongly related to game theory and potential games (see \cite{Monderer1996124}). In the single-agent case, this potential function coincides with the individual nonlinear map.
Similar extremum seeking vehicles were analyzed in \cite{zagsk} and \cite{zsk} by using averaging theory (see \cite{Khalil} and \cite{Sanders:2007uq}). The authors proposed various extremum seeking feedbacks for different vehicle dynamics and provided a local stability analysis for quadratic maps.
Using sinusoidal perturbations with vanishing gains, the authors of \cite{5400504} and \cite{msds1} were able to extend these results to prove almost sure convergence in the case of noisy measurements  of the map.
In a slightly different setup the authors of \cite{nesic} considered feedbacks which stabilize the extremum of a scalar, dynamic input-output map and established semi-global practical stability of the overall system under some technical assumptions. 
Multi-agent extremum seeking setups which use similar game-theoretic approaches can be found in \cite{Stankovics:2010fk}, where the agents seek a Nash equilibrium (see \cite{nash:1951}). The authors proved almost sure convergence of the scheme but without explicit consideration of the global stability properties. A closely related result, which considers the local stability of Nash equilibrium seeking systems, can be found in \cite{krstic:tac}.



Preliminary results of this work were published in \cite{Durr:2011uq} and \cite{Durr:2011kx} where the main proofs were omitted. Moreover, the results in this paper are more general.

\subsection{Organization}
The remainder of this paper is structured as follows. In Section 2 we illustrate the main idea using a simple example. In Section 3 we present theoretical results which link the stability properties of an input-affine system to its Lie bracket system. In Section 4 we apply these results to analyze stability properties of multi-agent extremum seeking systems. Finally, in Section 5 we illustrate the results with examples and give a conclusion in Section 6. 
\subsection{Notation}
\label{sec:preliminaries}
$\mathbb{N}_0$ denotes the set of positive integers including zero. $\mathbb{Q}_{++}$ denotes the set of positive rational numbers. The intervals of real number are denoted by $(a,b)=\{x\in\mathbb{R}:a<x<b\}$, $[a,b)=\{x\in\mathbb{R}:a\leq x<b\}$ and $[a,b]=\{x\in\mathbb{R}:a\leq x\leq b\}$. Let $f:\mathbb{R}^n\times\mathbb{R}^m\to\mathbb{R}^k$, then we write $f(\cdot,y)$ if we consider $f$ as a function of the first argument only and for all $y\in\mathbb{R}^m$. We denote by $C^n$ with $n\in\mathbb{N}_0$ the set of $n$ times continuously differentiable functions and by $C^\infty$ the set of smooth function.
The norm $|\cdot|$ denotes the Euclidian norm. 
The Jacobian of a continuously differentiable  function $b \in C^1: \mathbb{R}^{n}\to\mathbb{R}^{m}$
is denoted by
\begin{equation*}
\frac{\partial b(x)}{\partial x} := \begin{bmatrix}
\frac{\partial b_{1}(x)}{\partial x_{1}} & \ldots &  \frac{\partial b_{1}(x)}{\partial x_{n}} \\
\vdots & \ddots & \vdots\\
\frac{\partial b_{m}(x)}{\partial x_{1}} & \ldots & \frac{\partial b_{m}(x)}{\partial x_{n}}
\end{bmatrix}
\end{equation*}
and the gradient of a continuously differentiable function $J \in C^1:
\mathbb{R}^{n} \to \mathbb{R}$ is denoted by
$\nabla_{x}J(x) := \left[\frac{\partial J(x)}{\partial x_{1}},
\ldots, \frac{\partial J(x)}{\partial x_{n}}\right]^{\top}$. 
The Lie bracket of two vector fields $f,g : \mathbb{R}\times\mathbb{R}^n\to\mathbb{R}^n$ with $f(t,\cdot)$, $g(t,\cdot)$ being continuously differentiable is defined by $[f,g](t,x):= \frac{\partial g(t,x)}{\partial x}f(t,x) - \frac{\partial
f(t,x)}{\partial x}g(t,x)$. 
The $a$-neighborhood of a set $\mathcal{S}\subseteq\mathbb{R}^{n}$ with $a\in(0,\infty)$ is denoted by $\mathcal{U}^\mathcal{S}_a:=\{x \in \mathbb{R}^{n}:\inf_{y \in \mathcal{S}}|x-y|<a\}$. $\bar{\mathcal{U}}^\mathcal{S}_a$ denotes the closure of $\mathcal{U}^\mathcal{S}_a$. 
A function $u:\mathbb{R}\to\mathbb{R}$ is called measurable if it is Lebesgue-measurable. We use $s\in\mathbb{C}$ for the complex variable of the Laplace transformation if not indicated otherwise. 
\section{Main Idea}
One simple extremum seeking feedback for static maps is shown in Fig. \ref{fig:singleintsimp} (see also \cite{Krstic:2003fk} and \cite{zsk}).
Suppose that the function $f\in C^2:\mathbb{R}\to\mathbb{R}$ admits a local, strict maximum at $x^*$ and $\alpha,\omega\in(0,\infty)$. 
\begin{figure}[htpb]
\centering
\includegraphics[width=150px]{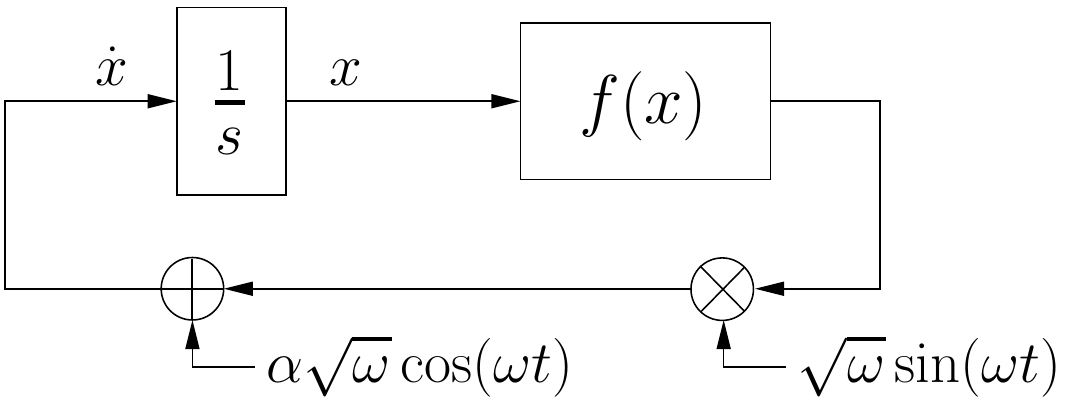}
\caption{Basic extremum seeking system}
\label{fig:singleintsimp}
\end{figure}

The extremum seeking system can be written as 
\begin{equation} \label{eq:intro1}
\dot{x} = \alpha\sqrt{\omega}\cos(\omega t)+ f(x)\sqrt{\omega}\sin(\omega t).
\end{equation}
The main idea is now to identify $\sin(\omega t)$ and $\cos(\omega t)$ as artificial inputs, i.e. $u_1(\omega t):=\cos(\omega t)$ and $u_2(\omega t):=\sin(\omega t)$. Thus, we obtain an input-affine system of the form
\begin{equation}
\dot{x}=b_1(x)\sqrt{\omega}u_1(\omega t) + b_2(x)\sqrt{\omega}u_2(\omega t)
\end{equation}
with $b_1(x)=\alpha$ and $b_2(x)=f(x)$.
Interestingly, if one computes the so called Lie bracket system involving $[b_1,b_2]$, i.e.
\begin{equation}\label{eq:intro2}
\dot{z} = \frac{1}{2}[b_1,b_2](z) = \frac{\alpha}{2}\nabla_z f(z), 
\end{equation}
then one sees that this system maximizes $f$. Having in mind, that trajectories resulting from sinusoidal inputs in \eqref{eq:intro1} can be approximated by trajectories of \eqref{eq:intro2} (see \cite{220059}, \cite{Kurzweil:1987fk}, \cite{573210}, \cite{H.-J.-Sussmann:1993vn}) allows us to establish a novel methodology to analyze extremum seeking systems.


The goal of this paper is to generalize this viewpoint to a larger class of extremum seeking systems. We derive a methodology which allows to analyze a broad class of extremum seeking systems by calculating their respective Lie bracket systems. The procedure can be summarized as follows: Write the extremum seeking system in input-affine form, calculate its corresponding Lie bracket system and prove asymptotic stability of the Lie bracket system which implies practical asymptotic stability for the extremum seeking system. 
\section{Lie Bracket Approximation for a Class of Input-Affine Systems}
In this section we consider a class of input-affine systems depending on a parameter and we deliver general results for approximating the trajectories of such systems by the trajectories of their respective Lie bracket systems. First, we state the definition of practical stability of a compact, invariant set for this class of systems. Second, we prove that their trajectories are approximated by the trajectories of their corresponding Lie bracket system for large values of the parameter. Third, we show how the stability properties of the input-affine system and the Lie bracket system are linked. The results in this section rely on a combination of results in \cite{220059}, \cite{Kurzweil:1987fk}, \cite{573210}, \cite{H.-J.-Sussmann:1993vn} and \cite{871771}, \cite{Moreau:2003fk}.

\subsection{Practical Stability}
In the following, we define the notion of practical stability which is closely related to Lyapunov stability and applies to differential equations depending on a parameter. Throughout the paper, we denote this parameter as $\omega$. For related literature about this concept we refer to \cite{871771}, \cite{nesic}, \cite{758492} and references therein. 

Let $x(\cdot):=x(\cdot;t_0,x_0,\omega)$ denote the solution of the differential equation
\begin{equation}
\label{eq:syspguas}
\dot{x}=f_\omega(t,x)
\end{equation}
through $x(t_0)=x_0$, where the vector field $f_\omega: \mathbb{R}\times\mathbb{R}^n\to\mathbb{R}^n$ depends on $\omega\in(0,\infty)$. 
\begin{defnt}
A compact set $\mathcal{S}\subseteq\mathbb{R}^n$ is said to be \textbf{practically uniformly stable} for \eqref{eq:syspguas} if for every $\epsilon\in(0,\infty)$ there exists a $\delta\in(0,\infty)$ and  $\omega_0\in(0,\infty)$ such that for all $t_0\in\mathbb{R}$ and for all  $\omega\in(\omega_0,\infty)$
\begin{equation}
\label{eq:pus}
x(t_0)\in \mathcal{U}_\delta^\mathcal{S} \Rightarrow x(t)\in \mathcal{U}_\epsilon^\mathcal{S}, t\in[t_0,\infty).
\end{equation}
\end{defnt}
\begin{defnt}
Let $\delta\in(0,\infty)$. A compact set $S\subseteq\mathbb{R}^n$ is said to be  \textbf{$\delta$-practically uniformly attractive} for \eqref{eq:syspguas} if for every $\epsilon\in(0,\infty)$ there exists a $t_f\in[0,\infty)$ and $\omega_0\in(0,\infty)$ such that for all $t_0\in\mathbb{R}$ and all $\omega\in(\omega_0,\infty)$
\begin{equation}
x(t_0)\in \mathcal{U}_\delta^\mathcal{S} \Rightarrow x(t)\in \mathcal{U}_\epsilon^\mathcal{S}, t\in[t_0+t_f,\infty).
\end{equation} 
\end{defnt}
\begin{defnt}
A compact set $\mathcal{S}\subseteq\mathbb{R}^n$ is said to be \textbf{locally practically uniformly asymptotically stable} for \eqref{eq:syspguas} if it is practically uniformly stable and there exists a $\delta\in(0,\infty)$ such that it is $\delta$-practically uniformly attractive.
\end{defnt}
\begin{defnt}
Let $\mathcal{S}\subseteq\mathbb{R}^n$ be a compact set. The solutions of \eqref{eq:syspguas} are said to be \textbf{practically uniformly bounded} if for every $\delta\in(0,\infty)$ there exists an $\epsilon\in(0,\infty)$ and $\omega_0\in(0,\infty)$ such that for all $t_0\in\mathbb{R}$ and for all $\omega\in(\omega_0,\infty)$
\begin{equation}
x(t_0)\in \mathcal{U}_\delta^\mathcal{S} \Rightarrow x(t)\in \mathcal{U}_\epsilon^\mathcal{S}, t\in[t_0,\infty).
\end{equation}
\end{defnt}
\begin{defnt}
A compact set $\mathcal{S}\subseteq\mathbb{R}^n$ is said to be \textbf{semi-globally practically uniformly asymptotically stable} for \eqref{eq:syspguas} if it is practically uniformly stable and for every $\delta\in(0,\infty)$ it is $\delta$-practically uniformly attractive. Furthermore the solutions of \eqref{eq:syspguas} must be practically uniformly bounded.
\end{defnt}

When \eqref{eq:syspguas} is independent of $\omega$ we omit the term ``practically'' in Definitions 1 -- 5 as well as ``semi'' in Definition 5. In this case, they are equivalent to the notion of stability in the sense of Lyapunov, we refer to e.g. \cite{Khalil}, \cite{871771} and \cite{Hale:1969kx}.
\subsection{Lie Bracket Approximation}
Throughout the paper, we consider the class of input-affine systems which can be written in the following form 
\begin{equation}\label{eq:input_affine_system}
\begin{split}
\dot{x} &= b_0(t,x)+\sum_{i=1}^{m}b_{i}(t,x)\sqrt{\omega}u_{i}(t,\omega t)
\end{split}
\end{equation}
with $x(t_0)=x_0\in\mathbb{R}^n$ and $\omega \in (0,\infty)$. Next, we define a differential equation, which we call the \emph{Lie bracket system} corresponding to \eqref{eq:input_affine_system}
\begin{equation}
\begin{split}
\dot{z} = b_0(t,z) +
\sum_{\substack{i=1\\j=i+1}}^m [b_{i},b_{j}](t,z)\nu_{ji}(t)
\label{eq:liebracket_system}
\end{split}
\end{equation}
with
\begin{equation}
\label{eq:nu_ij}
\nu_{ji}(t) = \frac{1}{T}\int_{0}^{T}u_{j}(t,\theta)\int_{0}^{\theta}u_{i}(t,\tau)d\tau
d\theta.
\end{equation}

\begin{remk}
If $u_i$ can be decomposed as $u_i(t,\omega t)=r_i(t)\tilde{u}_i(\omega t)$, $i=1,\ldots,m$, then \eqref{eq:input_affine_system} yields  $\dot{x} = b_0(t,x)$ $+\sum_{i=1}^{m}\tilde{b}_{i}(t,x)\sqrt{\omega}\tilde{u}_{i}(\omega t)$ with $\tilde{b}_i(t,x) = b_i(t,x)r_i(t)$. 
This is the usual setup in the existing literature (see e.g. \cite{Kurzweil:1987fk}, \cite{H.-J.-Sussmann:1993vn}). 
\end{remk}
We impose the following assumptions on $b_i$ and $u_i$: 
\begin{enumerate}[{A}1]
\item\label{ass:a1} $b_i \in C^2:\mathbb{R}\times\mathbb{R}^n\to\mathbb{R}^n$, $i=0,\ldots,m$.
\item\label{ass:a2} For every compact set $\mathcal{C}\subseteq\mathbb{R}^n$ there exist $A_1,\ldots,A_6\in[0,\infty)$ such that
$|b_i(t,x)|\leq A_1$, $|\frac{\partial b_i(t,x)}{\partial t}|\leq A_2$, $|\frac{\partial b_i(t,x)}{\partial x}|\leq A_3$, $|\frac{\partial^2 b_j(t,x)}{\partial t\partial x}|\leq A_4$, $|\frac{\partial [b_j,b_k](t,x)}{\partial x}|\leq A_5$, $|\frac{\partial [b_j,b_k](t,x)}{\partial t}|\leq A_6$
for all $x\in \mathcal{C}$, $t\in\mathbb{R}$, $i=0,\ldots,m$, $j=1,\ldots,m$, $k=j,\ldots,m$.
\item\label{ass:a3} $u_i:\mathbb{R}\times\mathbb{R}\to\mathbb{R}$, $i=1,\ldots,m$ are measurable functions. Moreover, for every $i=1,\ldots,m$ there exist constants $L_i,M_i\in(0,\infty)$ such that $|u_i(t_1,\theta)-u_i(t_2,\theta)| \leq L_i|t_1-t_2|$ for all $t_1,t_2\in\mathbb{R}$ and such that $\sup_{t,\theta\in\mathbb{R}}|u_i(t,\theta)|\leq M_i$. 
\item\label{ass:a4} $u_i(t,\cdot)$ is $T$-periodic, i.e. $u_i(t,\theta + T) = u_i(t,\theta)$, and has zero average, i.e. $\int_{0}^T u_i(t,\tau)d\tau = 0$, with $T\in(0,\infty)$ for all $t,\theta\in\mathbb{R}$, $i=1,\ldots,m$. 
\end{enumerate}
\begin{remk}
Assumption A1 is a regularity assumption on the vector fields, which are usually assumed to be smooth in the case of extremum seeking systems (see \cite{Krstic:2000kx} and \cite{nesic}).
\end{remk}
\begin{remk}
Assumption A2 means that expressions involving  $b_i$, $i=0,\ldots,m$ and their derivatives must be bounded uniformly in $t$. 
A similar assumption was made in Eq. (2.2), Section 2 in \cite{Kurzweil:1987fk}.
\end{remk}
\begin{remk}
Assumption A3 imposes measurability on $u_i$, $i=1,\dots,m$ which is necessary to establish existence of solutions of \eqref{eq:input_affine_system} (see Theorem \ref{thm:exuni} in Appendix \ref{app:exuni}). Alternatively, one could impose that the inputs $u_i$, $i=1,\ldots,m$ are continuous functions and argue using the existence and uniqueness theorem of Picard-Lindel\"of (see \cite{Coddington:1955fk}). However, this does not cover the case of piecewise continuous inputs, which might be interesting in certain applications, i.e. replacing the sinusoids with piecewise constant functions in the extremum seeking systems. 
\end{remk}
\begin{remk}
Similarly as in \cite{220059} we impose in Assumption A4 the $T$-periodicity and zero average of $u_i$, $i=1,\ldots,m$, which is common in the averaging literature but also in the literature dealing with Lie brackets. 
\end{remk}

Finally, we introduce a set $\mathcal{B}$ of initial conditions for \eqref{eq:liebracket_system} which have uniformly bounded solutions, i.e. there exists an $A\in(0,\infty)$ such that for all $t_0\in\mathbb{R}$ we have that
\begin{equation}\label{eq:defB}
z(t_0)\in\mathcal{B}\Rightarrow z(t)\in\mathcal{U}^0_A, t\in[t_0,\infty). 
\end{equation}
$\mathcal{B}$ is used in the proof of the main theorems and is crucial in order to assure existence of trajectories uniformly in $t_0$.

In the following, we state the main theorems which link stability properties of the systems in
\eqref{eq:input_affine_system} and \eqref{eq:liebracket_system}. The first theorem states that trajectories of \eqref{eq:input_affine_system} are approximated by trajectories of \eqref{eq:liebracket_system}. 
Related results are presented in \cite{220059}, \cite{573210} and \cite{758492}. 
However, we show for a larger class of inputs that the time interval of approximation can be made arbitrary large by choosing $\omega$ sufficiently large. We extend this result to infinite time-intervals and prove that the semi-global (local) practical uniform asymptotic stability of the input-affine system \eqref{eq:input_affine_system} follows from the global (local) uniform asymptotic stability of the corresponding Lie bracket system \eqref{eq:liebracket_system}. 
These results are stated in the second and third theorem which are similar to results in \cite{871771}. 
\begin{thm}
\label{thm:traj_approx}
Let Assumptions A\ref{ass:a1}--A\ref{ass:a4} be satisfied. Then for every bounded set $\mathcal{K}\subseteq \mathcal{B}$ with $\mathcal{B}$ as in \eqref{eq:defB}, for every $D\in(0,\infty)$ and for every $t_f\in(0,\infty)$, there exists an $\omega_0\in(0,\infty)$ such that for every $\omega\in(\omega_0,\infty)$, for every $t_0\in\mathbb{R}$ and every $x_0\in \mathcal{K}$ there exist solutions $x$ and $z$ of \eqref{eq:input_affine_system} and \eqref{eq:liebracket_system}  through $x(t_0)=z(t_0)=x_0$ which satisfy
\begin{equation}
|x(t)-z(t)|<D, \;\;\; t\in[t_0,t_0+t_f].
\end{equation}
\end{thm}

The proof of Theorem \ref{thm:traj_approx} uses similar arguments as in B.3, p. 1941 in \cite{Moreau:2003fk} but we consider more general inputs, which are characterized by Assumptions A3 and A4. The proof can be found in Appendix \ref{pf:traj_approx}.


\begin{thm}
\label{thm:liesystem_loc}
Let Assumptions A\ref{ass:a1}--A\ref{ass:a4} be satisfied and suppose that a compact set $\mathcal{S}$ is locally uniformly asymptotically stable for \eqref{eq:liebracket_system}. Then $\mathcal{S}$ is locally practically uniformly asymptotically stable for \eqref{eq:input_affine_system}. 
\end{thm}
The proof can be found in Appendix \ref{pf:liesystem_loc}.

\begin{thm}
\label{thm:liesystem_glob}
Let Assumptions A\ref{ass:a1}--A\ref{ass:a4} be satisfied and suppose that a compact set $\mathcal{S}$ is globally uniformly asymptotically stable for \eqref{eq:liebracket_system}. Then $\mathcal{S}$ is semi-globally practically uniformly asymptotically stable for \eqref{eq:input_affine_system}. 
\end{thm}

We omit the proof of Theorem \ref{thm:liesystem_glob} since it is already covered in \cite{871771} for the case of $\mathcal{S}$ being the origin. The proof directly carries over to compact sets $\mathcal{S}$ by replacing the Euclidian norm with a distance function to the set $\mathcal{S}$.

\begin{remk}
The results above only capture stability and not performance and do not deliver a systematic way for choosing $\omega$. The notion of practical stability only requires the existence of $\omega_0$ without explicitly considering a specific value. As indicated by Theorem \ref{thm:traj_approx} the choice of $\omega$ depends on the set of initial conditions $\mathcal{K}$, the distance $D$ and the time $t_f$. 
\end{remk}
\section{Lie Bracket Approximation of Extremum Seeking Systems}
In this section, we show how the results from the previous section can be applied to multi-agent extremum seeking systems. As indicated in Section 2, the procedure consists of writing the extremum seeking system in the input-affine form, calculating the corresponding Lie bracket system and finally concluding the respective stability properties of the extremum seeking system from the stability properties of the Lie bracket system by using Theorems \ref{thm:liesystem_loc} and \ref{thm:liesystem_glob}. 

In the following, we define a suitable framework for multi-agent extremum seeking systems. Suppose a group of $N$ agents tries to achieve a common goal which is defined as an extremum of a map $F$.
Specifically, we enumerate the agents using the superscript $i$. The position of agent $i$ is denoted by $\bar{x}^i=[x^i_1,x^i_2]^\top\in\mathbb{R}^2$. We define furthermore $\bar{x}:=[x^1_1,x^1_2,\ldots, x^N_1,x^N_2]^{\top}$ as the position vector of the overall system. Every agent is equipped with a specific extremum seeking feedback, which is defined below. We do not assume that all agents are seeking the extremum of the same map, but rather that each agent is equipped with an individual map $f^i:\mathbb{R}^{2N}\to\mathbb{R}$, $i=1,\ldots,N$, which also depends on the states of the other agents and satisfies
\begin{enumerate}
\item[B1] $f^i\in C^2$, $i=1,\ldots,N$.
\end{enumerate}
Furthermore, the individual maps have to satisfy the following assumption
\begin{enumerate}
\item[B2] \label{ass:b1} There exists a function $F\in C^1: \mathbb{R}^{2N}\to\mathbb{R}$ such that  
$\nabla_{{\bar{x}}^i}f^i({\bar{x}}) = \nabla_{{\bar{x}}^i}F({\bar{x}}),  i=1,\ldots,N, \bar{x}\in\mathbb{R}^{2N}$.
\end{enumerate}
These conditions implies, that if every agent moves into the direction of the gradient of its individual map $f^i$ then it also moves in the direction of the gradient of $F$. We call this a potential function. The goal of the multi-agent system is to find the minimum (maximum) of the common map $F$ by only seeking the minimum (maximum) of the individual map $f^i$. 

The following assumptions guarantee the existence of local (global) maxima of the potential function
\begin{enumerate}
\item[B3] \label{ass:b2} There exists a nonempty and compact set $\mathcal{S}_{loc}\subseteq\mathbb{R}^{2N}$ of strict local maxima and a $\delta\in(0,\infty)$ such that $F(\bar{x}^*)>F(\bar{x})$ for all $\bar{x}^*\in\mathcal{S}_{loc}$ and all $\bar{x}\in\mathcal{U}^{\mathcal{S}_{loc}}_\delta\backslash\mathcal{S}_{loc}$. Furthermore, $\nabla_{{\bar{x}}}F({\bar{x}})=0$ implies ${\bar{x}}\in \mathcal{S}_{loc}$ for all ${\bar{x}}\in\mathcal{U}^{\mathcal{S}_{loc}}_\delta$.
\item[B4] \label{ass:b3} There exists a nonempty and compact set $\mathcal{S}_{glob}=\{\bar{x}\in\mathbb{R}^{2N}: \bar{x}=\arg\max_{x\in\mathbb{R}^{2N}}F(\bar{x})\}$ of global maxima. Furthermore, $F(\bar{x})\to -\infty$ for $|\bar{x}|\to \infty$ and $\nabla_{{\bar{x}}}F({\bar{x}})=0$ implies ${\bar{x}}\in \mathcal{S}_{glob}$ for all ${\bar{x}}\in\mathbb{R}^{2N}$.
\end{enumerate}
This framework originates from game theory, where Assumption B2 formally defines a potential game with potential function $F$. We refer to \cite{Monderer1996124} for more information on potential games.
\begin{remk}
Under the assumptions above, the common goal can be formalized as the minimization (maximization) of the potential function $F$. There exist powerful tools to construct meaningful individual maps for a given potential function (see e.g. the approach using the so-called \emph{Wonderful Life Utility} in \cite{WolpertWLU}). 
The design should be done such that an optimization of the individual maps leads to an optimization of $F$, see \cite{Monderer1996124}. For this case, even though the utility functions are designed, they usually depend on some parameters or functions (e.g. environmental conditions, individual agents' properties) which are unknown a priori. A typical example for this scenario is the coverage control problem formulated as a potential game in \cite{4814554} and  \cite{Durr:2011uq}.
These aspects justify the usage of extremum seeking in this setup. For a specific application of the extremum seeking in a potential game framework we refer to \cite{Durr:2011uq}.
\end{remk}


In the next subsection, we show how the above framework above can be combined with extremum seeking agents. We saw in Section 2 that the trajectories of the extremum seeking system can be approximated by the trajectories of its corresponding Lie bracket system, which moves into the gradient direction of its individual map. We generalize this to the multi-agent case. If each agent is equipped with an extremum seeking feedback which drives it into the gradient direction of its individual map $f^i$, we expect with Assumption B2 that the overall system practically converges to an extremum of $F$. This is shown in the next subsection.

\subsection{Multi-Agent Extremum Seeking}
We show how extremum seeking can be applied to the above framework  assuming single-integrator agent dynamics. 

Consider the system in Fig. \ref{fig:singleint} which is motivated by a similar extremum seeking feedback as in \cite{zsk}. Since the agents move in the plane, there are two extremum seeking loops, one for each dimension. The perturbations are chosen to be sinusoidal, whose frequencies are chosen for each agent individually, as specified below. The high-pass filters $G^i(s)=\frac{s}{s+h^i}$, $i=1,\ldots,N$ are introduced since they provide better transient behavior by removing possible constant offsets of the individual maps $f^i$, $i=1,\ldots,m$. They introduce an additionally degree of freedom, but do not influence the stability of the overall system, as it can be seen in the proofs of Theorems \ref{thm:mas_single_int_loc} and \ref{thm:mas_single_int_glob}. 

Define $\bar{x}_e:=[x^1_e,$ $\ldots,x^N_e]^{\top}$ and $x:=[\bar{x}^\top, \bar{x}_e^\top]^{\top}$ with $x^i_e$ denoting the state of the filter $G^i(s)=\frac{s}{s+h^i}$, i.e. in state space form we have $\dot{x}^i_e=-x^i_eh^i+u^i$ and $y^i=-x^i_eh^i+u^i$ with $u^i=f^i(\bar{x})$. 

\begin{figure}[t]
\centering
\includegraphics[width=180pt]{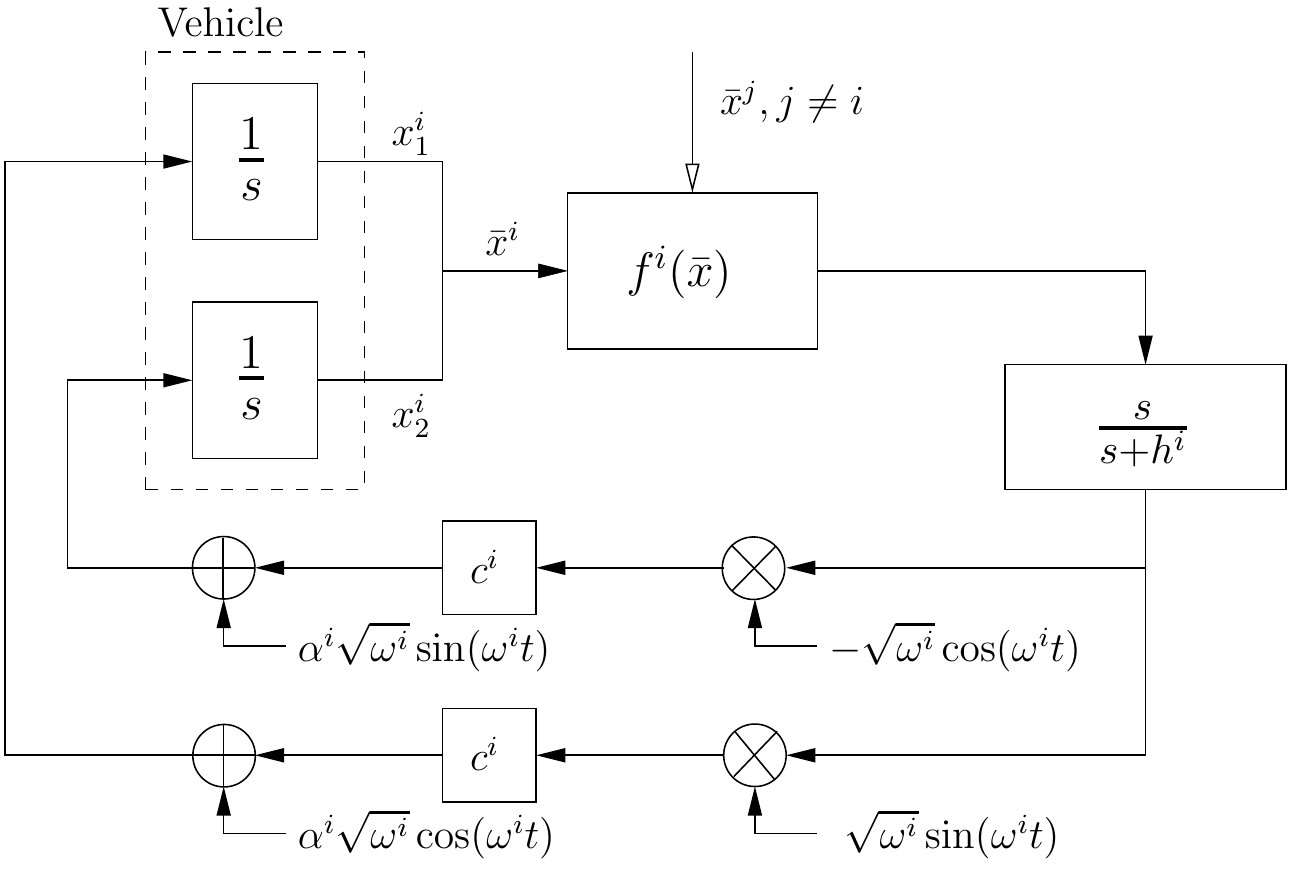}
\caption{Single-integrator dynamics}
\label{fig:singleint}
\end{figure}

The differential equations describing the dynamics of agent $i$ are given by
\begin{align}
\nonumber\dot{x}^i_{1} =& c^i\bigl(f^i(\bar{x})-x^i_eh^i\bigr)\sqrt{\omega^i}u^i_{1}(\omega^i t) + \alpha^i\sqrt{\omega^i}u^i_{2}(\omega^i t) \\
\nonumber\dot{x}^i_{2} =& -c^i\bigl(f^i(\bar{x})-x^i_eh^i\bigr)\sqrt{\omega^i}u^i_{2}(\omega^i t) + \alpha^i\sqrt{\omega^i}u^i_{1}(\omega^i t) \\
\label{eq:originalsystemsingleint}\dot{x}^i_e =& -x^i_eh^i + f^i(\bar{x})
\end{align}
with $u^i_{1}(\omega^i t)=\sin(\omega^i t)$, $u^i_{2}(\omega^i t)=\cos(\omega^i t)$. 

We need an additional assumption for the multi-agent case concerning the parameter $\omega$. We see in the proof of the next theorem that if the following assumption is satisfied, then some of the $\nu_{ji}$ in \eqref{eq:nu_ij} vanish in the corresponding Lie bracket system. This can be assured by assuming
\begin{enumerate}
\item[B5] $\omega^{i} = a^{i}\omega$ and $a^{i}\neq a^{j}, i\neq j$, $a^{i} \in \mathbb{Q}_{++}$, $\omega \in
(0,\infty)$, $h^{i}, \alpha^i, c^i \in (0,\infty)$, $i,j=1,\ldots,N$.
\end{enumerate}
Since the high-pass filter $\frac{s}{s+h^i}$ introduces an additional state $x^i_e$, which has also to be taken into account in the analysis, we denote by 
\begin{equation}
\begin{split}
\mathcal{E}^\mathcal{S}:=\{&\bar{x}_e\in\mathbb{R}^N: \\
 &\bar{x}_e=\biggl[\frac{f^1(\bar{x})}{h^1},\ldots,\frac{f^N(\bar{x})}{h^N}\biggr]^\top, \bar{x}\in\mathcal{S}\}
\end{split}
\end{equation} 
with $\mathcal{S}$ is either $\mathcal{S}_{loc}$ or $\mathcal{S}_{glob}$, the set which is shown to be attractive for the filter states $x^i_e$, $i=1,\ldots,N$.
\begin{thm}\label{thm:mas_single_int_loc}
Consider a multi-agent system with $N$ agents, each one having dynamics given by \eqref{eq:originalsystemsingleint}. Let Assumptions B1 to B3 and B5 be satisfied, then the set $\mathcal{S}_{loc}\times\mathcal{E}^{\mathcal{S}_{loc}}$ is locally practically uniformly asymptotically stable for the overall system with state $[\bar{x}^\top,\bar{x}_e^\top]^\top$. 
\end{thm}
\begin{proof}
The proof can be split up into three steps. In the first step, we rewrite the system in the input-affine form. In the second step, we calculate the corresponding Lie bracket system and in the third step, we prove uniform asymptotic stability of the Lie bracket system. Theorem \ref{thm:liesystem_loc} then allows to conclude practical asymptotic stability for the original system.

In the first step, we rewrite the overall system with state $x=[\bar{x}^\top,\bar{x}_e^\top]^\top$, where each component is described by the differential equations given in \eqref{eq:originalsystemsingleint}, as input-affine system of the form
\begin{equation}
\label{eq:multi_agent_input_affine_old}
\begin{split}
\!\dot{{x}}
=&\sum_{i=1}^N
b^i_{0}({x})
+
b^i_{1}({x})\sqrt{\omega^i}\underbrace{\sin(\omega^i t)}_{=:u^i_{1}(\omega^i t)}
\\
&+
b^i_{2}({x})\sqrt{\omega^i}\underbrace{\cos(\omega^i t)}_{=:u^i_{2}(\omega^i t)}
\end{split}
\end{equation}
with $b^i_0, b^i_1, b^i_2$ having non-zero entries only at positions corresponding to agent $i$ and zeros elsewhere, i.e. $b^i_{0}(x)=[0,\ldots,0,0,0,-x^i_e+f^i(\bar{x}),0,\ldots,0]^\top$, $b^i_{1}(x)=[0,\ldots,0,$ $c^i(f^i(\bar{x})-x^i_eh^i),$  
$\alpha^i, 0, 0,\ldots,0 ]^\top$, $b^i_{2}(x)=[0,\ldots,0,
\alpha^i,$ $-c^i(f^i(\bar{x})-x^i_eh^i), 0, 0,\ldots,0 ]^\top$.
 
Note that due to Assumption B5 we have that $a_i$ can be written as $a_i=\frac{p_i}{q_i}$ with $p_i,q_i\in\mathbb{N}$ and define $q:=\prod_{i=1}^{N}q_{i}$ and $\tilde{\omega}=\frac{\omega}{q}$. Thus, $a_i\omega =\frac{p_i}{q_i}\omega=p_{i}\prod_{j\neq
i}q_{j}\tilde{\omega}=n^i\tilde{\omega}$, $i=1,\ldots, N$ and $j=1,2$ and for $n^i:=p_{i}\prod_{j\neq i}q_{j}\in\mathbb{N}$. 
We rewrite \eqref{eq:multi_agent_input_affine_old} as follows
\begin{equation}
\label{eq:multi_agent_input_affine_new}
\begin{split}
\dot{{x}}
=&\sum_{i=1}^N
b^i_{0}({x})
+
b^i_{1}({x})\sqrt{n^i}\sqrt{\tilde{\omega}}u^i_{1}(n^i\tilde{\omega} t)
\\
&+
b^i_{2}({x})\sqrt{n^i}\sqrt{\tilde{\omega}}u^i_{2}(n^i\tilde{\omega} t).
\end{split}
\end{equation}
It can directly be seen that $u^i_k(n^i{\theta})\in\{\sin(n^i{\theta}),\cos(n^i{\theta})\}$ are also $2\pi$-periodic in $n^i\tilde{\omega} t$ for $i=1,\ldots, N$ and $k=1,2$ and for $n^i\in\mathbb{N}$.

In the second step, we calculate the corresponding Lie bracket system as defined in \eqref{eq:liebracket_system}. Define $\bar{z}:=[z^1_{1}, z^1_{2},\ldots, z^N_{1}, z^N_{2}]^{\top}$, $\bar{z}_e:=[z^1_e,\ldots,z^N_e]^{\top}$ and
$z:=[\bar{z}^\top,\bar{z}_e^\top]^{\top}$ 
and $\nu^{i,j}_{k,l} = \frac{1}{2\pi}\int_0^{2\pi}u^i_k(n^i\tau)\int_0^\tau u^j_l(n^j\theta)d\theta d\tau$ which are constant for all $i,j=1,\ldots,N$ and $k,l=1,2$.

The crucial point now is that some Lie brackets in the differential equation of the overall system vanish due to the choice of different parameters $\omega^i$ for the agents. We obtain using  Lemma \ref{lem:mas_simplification} (see Appendix \ref{appendix:proofThm1}) that $\nu^{i,j}_{k,l}=-\frac{1}{2n^i}$ for all $n^i= n^j$ and $k=l$ and $\nu^{i,j}_{k,l}=0$ otherwise. Thus, the Lie bracket system simplifies to
\begin{align}
\begin{split}
\label{eq:liemultagentsingleint}\dot{{z}}=&\; \sum_{i=1}^N b^i_{0}({z}) -\frac{1}{2n^i}[\sqrt{n^i}{b}^i_{1},\sqrt{n^i}{b}^i_{2}]({z}) \\
=&\;\sum_{i=1}^N b^i_{0}({z}) -\frac{1}{2}[{b}^i_{1},{b}^i_{2}]({z}).
\end{split}
\end{align}
Explicitly, for the states of agent $i$ we obtain
\begin{align} 
\nonumber\dot{z}^i_{1} =& \frac{1}{2}\bigl(c^i\alpha^i\nabla_{z^i_1}f^i(\bar{z})- {c^i}^{2}\nabla_{{z}^i_2}f^i(\bar{z})\bigl(f^i(\bar{z})-z^i_e h)\bigr) \\
\nonumber\dot{z}^i_{2} =&  \frac{1}{2}\bigl(c^i\alpha^i\nabla_{z^i_2}f^i(\bar{z})+{c^i}^{2}\nabla_{{z}^i_1}f^i(\bar{z})\bigl(f^i(\bar{z})-z^i_e h)\bigr) \\
\label{eq:lie_sys_ma}\dot{z}^i_e =&-z^i_eh^i+f^i(\bar{z}).
\end{align}

In the third step, we prove uniform asymptotic stability of the set ${\mathcal{S}_{loc}}\times\mathcal{E}^{\mathcal{S}_{loc}}$ for \eqref{eq:liemultagentsingleint}. We first need to show existence of the solutions of \eqref{eq:liemultagentsingleint} on $[t_0,\infty)$ for all $t_0\in\mathbb{R}$.
Note that the vector field in \eqref{eq:liemultagentsingleint} is independent of $t$ and continuously differentiable in $z$. The existence and uniqueness theorem by Picard-Lindel\"of  (see \cite{Coddington:1955fk}) guarantees that there exist a time $t_f\in(0,\infty)$ and a solution of $z:\mathbb{R}\to\mathbb{R}^{3N}$ defined on $[t_0,t_0+t_f)$ for all $t_0\in\mathbb{R}$. Note furthermore, that with $h^i\in(0,\infty)$ in Assumption B5, the differential equation for $z^i_e$, i.e. $\dot{z}^i_e=-h^iz^i_e +u$ with $u=f^i(\bar{z})$ in \eqref{eq:lie_sys_ma} is linear and its origin is exponentially stable for $u=0$. 
Thus if $f^i(\bar{z}(t))$ is bounded then $z^i_e(t)$ exists and is bounded with gain $\frac{1}{h^{i}}$ for all $i=1,\ldots,N$, for all $t_0\in\mathbb{R}$ and for all $t\in[t_0,\infty)$. 
Suppose now that ${\mathcal{S}_{loc}}$ is uniformly asymptotically stable for $\bar{z}$, 
then it can be shown that the set $\mathcal{E}^{\mathcal{S}_{loc}}$ is uniformly asymptotically stable for $z^i_e$, $i=1,\ldots,N$. Therefore, the set ${\mathcal{S}_{loc}}\times\mathcal{E}^{\mathcal{S}_{loc}}$ is uniformly asymptotically stable for the overall system $[\bar{z}^\top,\bar{z}_e^\top]^\top$.

It is left to show that the set ${\mathcal{S}_{loc}}$ is uniformly asymptotically stable for $\bar{z}$.
Choose $V:=-F$ which is due to Assumption B3 a valid Lyapunov function in $\mathcal{U}^{\mathcal{S}_{loc}}_\delta$. Observe that due to Assumption B2 we have that $\nabla_{\bar{z}^i}f^i(\bar{z}) = \nabla_{\bar{z}^i}F(\bar{z}), i=1,\ldots,N$ and thus
\begin{equation}\label{eq:dotV_singleint_ma}
\begin{split}
\dot{V} = -\sum_{i=1}^N\frac{c^i\alpha^i}{2}\biggl(&\nabla_{{z}^i_1}F(\bar{z})^\top \nabla_{{z}^i_1}F(\bar{z})\\
&+\nabla_{{z}^i_2}F(\bar{z})^\top \nabla_{{z}^i_2}F(\bar{z})\biggr).
\end{split}
\end{equation}
Due to $c^i,\alpha^i\in(0,\infty)$, $i=1,\ldots,N$ in Assumption B5, we know that $V(\bar{z}(t))$ is decreasing along the trajectories of $\bar{z}(t)$ for all $\bar{z}(t_0)\in\mathcal{U}^{\mathcal{S}_{loc}}_\delta$, all $t_0\in\mathbb{R}$ and all $t\in[t_0,t_0+t_f)$. We conclude that $|\bar{z}(t)|$ is bounded and therefore all $f^i(\bar{z}(t))$, $i=1,\ldots,N$, are bounded for all $\bar{z}(t_0)\in\mathcal{U}^{\mathcal{S}_{loc}}_\delta$, all $t_0\in\mathbb{R}$ and all $t\in[t_0,\infty)$. 
Thus, $z(t)=[\bar{z}(t)^\top, \bar{z}_e(t)^\top]^\top$ exists for all $t_0\in\mathbb{R}$, for all $z(t_0)\in\mathcal{U}^{\mathcal{S}_{loc}}_\delta$ and for all $t\in[t_0,\infty)$.
Furthermore, we conclude with \eqref{eq:dotV_singleint_ma} and Assumption B3 that the set $\mathcal{S}_{loc}$ is locally uniformly asymptotically stable for the subsystem $\bar{z}=[z^1_{1}, z^1_{2},\ldots, z^N_{1}, z^N_{2}]^{\top}$ in \eqref{eq:liemultagentsingleint}. 

Note that due to Assumption B1 and the fact that $u^i_k(n^i\tilde{\theta})\in\{\sin(n^i\tilde{\theta}),\cos(n^i\tilde{\theta})\}$ for $i=1,\ldots, N$ and $k=1,2$ we conclude that Assumptions A\ref{ass:a1} to A\ref{ass:a4} are satisfied. Thus, with Theorem \ref{thm:liesystem_loc} the set $\mathcal{S}_{loc}\times\mathcal{E}^{\mathcal{S}_{loc}}$ is locally practically uniformly asymptotically stable for the overall system with state $[\bar{x}^\top,\bar{x}_e^\top]^\top$.
\end{proof}
\begin{thm}\label{thm:mas_single_int_glob}
Consider a multi-agent system with $N$ agents, each one having dynamics given by \eqref{eq:originalsystemsingleint}. Let Assumptions B1, B2, B4 and B5 be satisfied, then the set $\mathcal{S}_{loc}\times\mathcal{E}^{\mathcal{S}_{glob}}$ is semi-globally practically uniformly asymptotically stable for the overall system with state $[\bar{x}^\top,\bar{x}_e^\top]^\top$. 
\end{thm}
\begin{proof}
If Assumption B4 is satisfied then $\mathcal{S}_{glob}$ is a connected set containing the global maximum of $F$. Furthermore, $F$ is radially unbounded and with \eqref{eq:dotV_singleint_ma} we see that if $\dot{V}(\bar{z})=0$ implies $\bar{z}(t)\in\mathcal{S}_{global}$. Thus, we conclude that $\mathcal{S}_{loc}\times\mathcal{E}^{\mathcal{S}_{glob}}$ is globally uniformly asymptotically stable for \eqref{eq:liemultagentsingleint} and thus with Theorem \ref{thm:liesystem_glob}, it is semi-globally practically uniformly asymptotically stable for the overall system with state $[\bar{x}^\top,\bar{x}_e^\top]^\top$.
\end{proof}

In the following, we analyze the same setup as before but replace the single-integrator dynamics with unicycle dynamics as shown in Fig.~\ref{fig:unicycle}. The setup is motivated by \cite{zagsk}.

\begin{figure}[htpb]
\centering
\includegraphics[width=180pt]{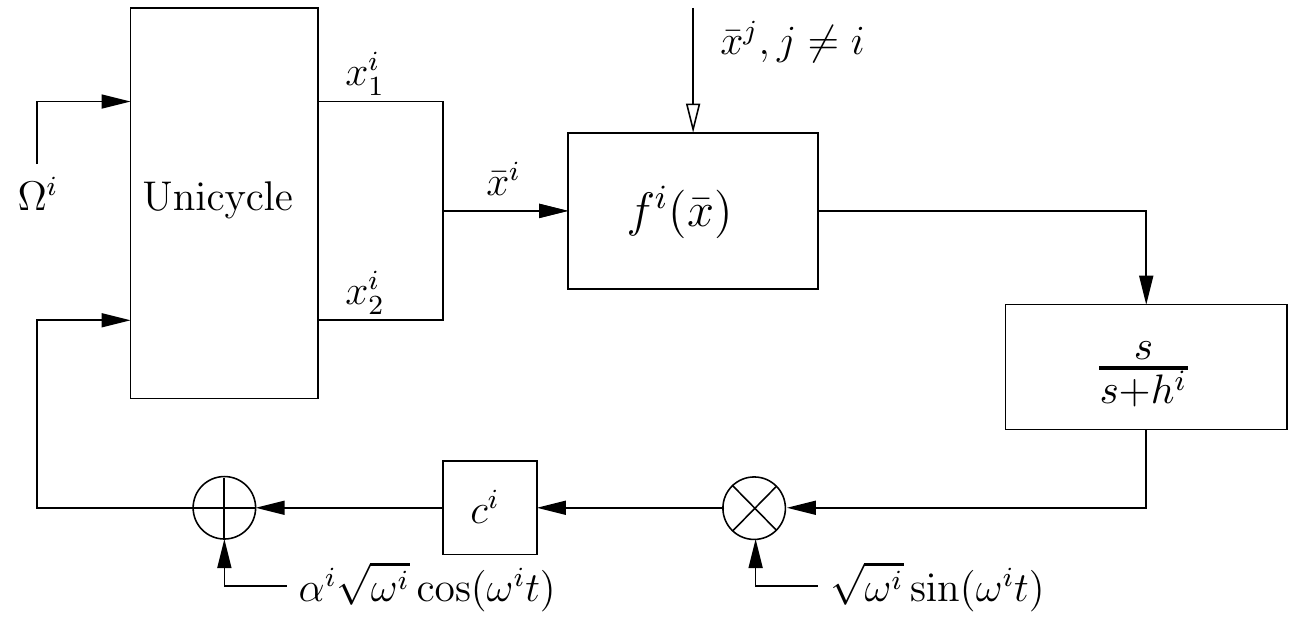}
\caption{Unicycle dynamics}
\label{fig:unicycle}
\end{figure}

Let us consider the unicycle model for each agent given by the
equations
\begin{equation} \label{eq:uni}
\begin{split}
\dot{x}^i_{1} &= u^i\cos(x_\theta^i), \;\; \dot{x}^i_{2} = u^i\sin(x_\theta^i), \;\; \dot{x}^{i}_{\theta} = v^i.
\end{split}
\end{equation}
The extremum seeking feedback controls only the forward velocity
of the vehicle, whereas the angular velocity is constant, so that
the inputs to each vehicle are
$u^i(t, x)=(c^{i}(f^i(\bar{x})-x^{i}_eh^{i})\sqrt{\omega^i}\sin(\omega^{i}
t)+\alpha^{i}\sqrt{\omega^i}\cos(\omega^{i} t))$ and
$v^i=\Omega^{i}$. We assume that $x^{i}_{\theta}(t_0)=0$ and for all $i=1,\ldots,N$ and
\begin{enumerate}
\item[B6] $\Omega^{i} = d^{i}\Omega$ with $d^{i} \in \mathbb{Q}_{++}$, $\Omega \in
\mathbb{R}\backslash \{0\}$.
\end{enumerate}
\begin{remk}
It becomes clear in the proof that the corresponding vector field of the Lie bracket system is time-varying and vanishes at discrete points in time. Assumption B6 assures that the vector field is periodic, so that a LaSalle-like argument can be used in order to prove uniform asymptotic stability. Note that the $\Omega^i$'s can be equal, whereas the $\omega^i$'s must be different for all agents.
\end{remk}
By substituting the expressions for the inputs into \eqref{eq:uni} and replacing $x^{i}_{\theta}(t)=\Omega^i t$ we obtain
\begin{equation}\label{eq:unicycle_1}
\begin{split}
\dot{x}^i_{1} = \biggl(c^{i}(f^i(\bar{x})-&x^i_{e}h^{i})\sqrt{\omega^{i}}u^i_{1}(\omega^i t)\\
&+\alpha^{i}\sqrt{\omega^{i}}u^i_{2}(\omega^i t)\biggr)\cos(\Omega^i t) \\
\dot{x}^i_{2} = \biggl(c^{i}(f^i(\bar{x})-&x^i_e h^i)\sqrt{\omega^{i}}u^i_{1}(\omega^i t)\\
&+\alpha^{i}\sqrt{\omega^{i}}u^i_{2}(\omega^i t)\biggr)\sin(\Omega^i t) \\
\dot{x}^i_e = -x^i_eh_{i}+f^i&(\bar{x})
\end{split}
\end{equation}
with $u^i_{1}(\omega^i t)=\sin(\omega^i t)$, $u^i_{2}(\omega^i t)=\cos(\omega^i t)$.

\begin{thm}\label{thm:mas_unicycle_loc}
Consider a multi-agent system with $N$ agents, each one having dynamics given by \eqref{eq:unicycle_1}. Let Assumptions B1 to B3, B5 and B6 be satisfied, then the set $\mathcal{S}_{loc}\times\mathcal{E}^{\mathcal{S}_{loc}}$ is locally practically uniformly asymptotically stable for the overall system with state $[\bar{x}^\top,\bar{x}_e^\top]^\top$. 
\end{thm}
\begin{proof}
The proof goes along the same lines as the proof of Theorem \ref{thm:mas_single_int_loc}. 
In the first step, we rewrite the overall system as input-affine system
\begin{equation}
\label{eq:multi_agent_input_affine_unicycle}
\begin{split}
\dot{{x}}
=\sum_{i=1}^N&
b^i_{0}({x})
+
b^i_{1}(t,{x})\sqrt{\omega^i}u^i_{1}(\omega^i t)
\\
&+
b^i_{2}(t,{x})\sqrt{\omega^i}u^i_{2}(\omega^i t),
\end{split}
\end{equation}
where $b^i_0, b^i_1, b^i_2$ have non-zero entries only at the positions corresponding to agent $i$ and zeros elsewhere, i.e. 
$b^i_0({x})=[0,\ldots, 0,0,0,-x^i_eh_{i}+f^i(x),0,\ldots,0]^\top$, $b^i_1(t,{x})=[0,\ldots,$ $(c^{i}(f^i(\bar{x})-x^i_{e}h^{i}))\cos(\Omega^i t), (c^{i}(f^i(\bar{x})-x^i_{e}h^{i}))$ $\sin(\Omega^i t),$ $ 0, 0, \ldots, 0]^\top$ and $b^i_2(t,{x})=[0,\ldots, \alpha^i\cos(\Omega^i t),$ $\alpha^i\sin(\Omega^i t),$ $ 0, 0, \ldots,0]^\top$. 

Note that due to Assumption B5 $a_i$ can be written as $a_i=\frac{p_i}{q_i}$ with $p_i,q_i\in\mathbb{N}$ and define $q:=\prod_{i=1}^{N}q_{i}$ and $\tilde{\omega}=\frac{\omega}{q}$. Thus, $a_i\omega =\frac{p_i}{q_i}\omega=p_{i}\prod_{j\neq
i}q_{j}\tilde{\omega}=n^i\tilde{\omega}$, $i=1,\ldots, N$ and $j=1,2$ and for $n^i:=p_{i}\prod_{j\neq i}q_{j}\in\mathbb{N}$. We rewrite \eqref{eq:multi_agent_input_affine_unicycle} as follows
\begin{equation}
\label{eq:multi_agent_input_affine}
\begin{split}
\dot{\bar{x}}
=\sum_{i=1}^N&
b^i_{0}({x})
+
{b}^i_{1}(t,{x})\sqrt{n^i}\sqrt{\tilde{\omega}}u^i_{1}(n^i\tilde{\omega} t)
\\
&+
{b}^i_{2}(t,{x})\sqrt{n^i}\sqrt{\tilde{\omega}}u^i_{2}(n^i\tilde{\omega} t).
\end{split}
\end{equation}
In the second step, we calculate the corresponding Lie bracket system as it was defined in \eqref{eq:liebracket_system}, 
\begin{equation}
\label{eq:liemultagentunicycle}\dot{{z}}=\sum_{i=1}^N b^i_{0}({z}) -\frac{1}{2}[{b}^i_{1},{b}^i_{2}](t,{z}).
\end{equation}
By the same reasoning as in the proof of Theorem \ref{thm:mas_single_int_loc}, this yields for the state of agent $i$
\begin{equation}
\begin{split}
\dot{z}^i_{1} =& \frac{1}{2}(c^{i}\alpha^{i}\nabla_{z^i_{1}}f^i(\bar{z})\cos^{2}(\Omega^{i}t)\\
&+c^i\alpha^{i}\nabla_{{z}^i_{2}}f^i(\bar{z})\cos(\Omega^{i}t)\sin(\Omega^{i}t))) \\
\dot{z}^i_{2} =& \frac{1}{2}(c^{i}\alpha^{i}\nabla_{{z}^i_{2}}f^i(\bar{z})\sin^{2}(\Omega_{i}t)\\
&+c^i\alpha^{i}\nabla_{{z}^i_{1}}f^i(\bar{z})\cos(\Omega^{i}t)\sin(\Omega^{i}t))) \\
\dot{z}^{i}_e =&-z^i_eh^i+f^i(\bar{z}).
\label{eq:lie_sys_unicycle}
\end{split}
\end{equation}

In the third step, we prove uniform asymptotic stability of the set $\mathcal{S}_{loc}\times\mathcal{E}^{\mathcal{S}_{loc}}$ for the Lie bracket system of \eqref{eq:multi_agent_input_affine}.
Due to Assumption B3 we exploit the function $V:=-F$ as a Lyapunov function candidate which is valid in $\mathcal{U}^{\mathcal{S}_{loc}}_\delta$. Observe that due to Assumption B2 we have that $\nabla_{\bar{z}^i}f^i(\bar{z}) = \nabla_{\bar{z}^i}F(\bar{z}), i=1,\ldots,N$ and thus
\begin{align} 
\nonumber \dot{V} =& -\sum_{i=1}^N\frac{c^i\alpha^i}{2} (\nabla_{{z}^i_1}F(\bar{z})\cos(\Omega^i t) + \nabla_{{z}^i_2}F(\bar{z})\sin(\Omega^i t))^{\top}\\
\label{eq:lyap_uni}&\cdot (\nabla_{{z}^i_1}F(\bar{z})\cos(\Omega^i t) + \nabla_{{z}^i_2}F(\bar{z})\sin(\Omega^i t)).
\end{align}
We have that $c^i, \alpha^i\in(0,\infty)$, $i=1,\ldots,N$ from Assumption B5, and thus $\dot{V}$ is negative semi-definite. Observe that the vector field in \eqref{eq:lie_sys_unicycle}  is
time-varying and there are time-instances where $\dot{\bar{z}}(t)=0$, but which are not steady-states for the system. Next, we make use of Assumption B6, which assures the existence of $k^i, l^i \in \mathbb{N}$, $i=1,\ldots,N$ such that $d^i = \frac{k^i}{l^i}$. One can verify that the vector field of the overall system \eqref{eq:liemultagentunicycle} consisting of $N$ agents with system equations as in \eqref{eq:lie_sys_unicycle}, is $T$-periodic with $T=\frac{2\pi}{\Omega}\prod_{i=1}^N l^i$. We can now use Theorem 4 in \cite{LaSalle:1962kl} which is LaSalle's Invariance Principle for periodic vector fields and conclude uniform asymptotic stability. It is left to show that no trajectory of \eqref{eq:liemultagentunicycle} can stay identically in the set where $\dot{V}(\bar{z})=0$ except for $\bar{z}\in \mathcal{S}_{loc}$. To see this, observe that the summands of $\dot{V}$ can only be equal to zero if
$\nabla_{{z}^i_{1}}F(\bar{z}(t))\cos(\Omega^{i}t)+
\nabla_{{z}^i_2}F(\bar{z}(t))\sin(\Omega^{i}t)=0$, $i=1,\ldots,N$. On the set $\dot{V}(\bar{z})=0$ the differential equation yields $\dot{z}^i_{1}=\dot{z}^i_{2}=0$ and therefore
${z}^i_{1}(t)=\text{const.}$ and ${z}^i_{2}(t)=\text{const.}$. Thus
$\nabla_{{z}^i_{1}}F(\bar{z}(t))=\text{const.}$ and
$\nabla_{{z}^i_{2}}F(\bar{z}(t)))=\text{const.}$. But there are
no constants $a,b\in\mathbb{R}$ such that
$a\cos(\Omega^{i} t) + b\sin(\Omega^{i} t) = 0$ for all $t\in[t_0,\infty)$ except $a=b=0$ and therefore $\nabla_{{z}^i_{1}}F(\bar{z}(t))
=\nabla_{{z}^i_{2}}F(\bar{z}(t))= 0$. We conclude that the set $\mathcal{S}_{loc}$ is locally uniformly asymptotically stable for the subsystem $\bar{z}$ in \eqref{eq:lie_sys_unicycle}. Observe furthermore, that due to $h^i\in(0,\infty)$ in Assumption B5, the differential equation for $z^i_e$, i.e. $\dot{z}^i_e=-h^iz^i_e +u$ with $u=f^i(\bar{z})$ in \eqref{eq:lie_sys_ma} is linear and its origin is exponentially stable for $u=0$. Thus if $f^i(\bar{z}(t))$ is bounded then $z^i_e(t)$ exists and is bounded with gain $\frac{1}{h^{i}}$ for all $i=1,\ldots,N$, for all $t_0\in\mathbb{R}$ and for all $t\in[t_0,\infty)$. Therefore, the set ${\mathcal{S}_{loc}}\times\mathcal{E}^{\mathcal{S}_{loc}}$ is uniformly asymptotically stable for the overall system $[\bar{z}^\top,\bar{z}_e^\top]^\top$.

Note that due to Assumption B1 and the fact that $u^i_k(n^i\tilde{\theta})\in\{\sin(n^i\tilde{\theta}),\cos(n^i\tilde{\theta})\}$ for $i=1,\ldots, N$ and $k=1,2$ we conclude that Assumptions A\ref{ass:a1} to A\ref{ass:a4} are satisfied. Thus, with Theorem \ref{thm:liesystem_loc} the set  ${\mathcal{S}_{loc}}\times\mathcal{E}^{\mathcal{S}_{loc}}$  is locally practically uniformly asymptotically stable for the overall system with state $[\bar{x}^\top,\bar{x}_e^\top]^\top$.
\end{proof}
\begin{thm}\label{thm:mas_unicycle_glob}
Consider a multi-agent system with $N$ agents, each one having dynamics given by \eqref{eq:unicycle_1}. Let Assumptions B1, B2 and B4 to B6 be satisfied, then the set  ${\mathcal{S}_{loc}}\times\mathcal{E}^{\mathcal{S}_{glob}}$  is semi-gobally practically uniformly asymptotically stable for the overall system with state $[\bar{x}^\top,\bar{x}_e^\top]^\top$. 
\end{thm}
The proof uses the same argumentation as the proof of Theorem \ref{thm:mas_single_int_glob}.
\section{Discussion}

\subsection{Relationship to Averaging Methods}
There is a close relationship between the results herein and averaging theory. The Lie bracket system in \eqref{eq:liebracket_system} can be seen as the averaged system of \eqref{eq:input_affine_system}. In order to use averaging theory, the system must be in the following form (see Eq. (10.23) on p. 404 in \cite{Khalil})
\begin{equation}\label{eq:averaging}
\frac{dx}{d\tau}=\epsilon b(\tau,x,\epsilon)
\end{equation}
with $x(\tau)\in\mathbb{R}^n$, $\epsilon\in(0,\infty)$ and $b\in C^2:\mathbb{R}\times\mathbb{R}^n\times\mathbb{R}\to\mathbb{R}^n$ where $b(\cdot,x,\epsilon)$ is $T$-periodic with $T\in(0,\infty)$. The associate averaged system is given by
\begin{equation}\label{eq:averaging_averaged}
\frac{dz}{d\tau}=\epsilon b_T(z)
\end{equation}
with $b_T(z) = \frac{1}{T}\int_{0}^Tb(s,z,0)ds$.

Standard averaging can not be applied directly to \eqref{eq:input_affine_system}. We show this with a simple calculation. After rescaling time $\tau=\omega t$ and by setting $\epsilon=\frac{1}{\omega}$ we obtain
\begin{equation}\label{eq:input_affine_trsfd}
\begin{split}
\frac{dx}{d\tau}&=\epsilon \biggl(b_0(\epsilon \tau,x)+\frac{1}{\sqrt{\epsilon}}\sum_{i=1}^{m}b_{i}(\epsilon \tau,x)u_{i}(\epsilon \tau, \tau)\biggr). 
\end{split}
\end{equation}
Since $\frac{1}{\sqrt{\epsilon}}$ appears in the vector field of \eqref{eq:input_affine_trsfd}  the vector field is not twice continuously differentiable and $b(\tau,z,0)$ does not exist. Thus the integral $\frac{1}{T}\int_{0}^Tb(s,z,0)ds$ does not exist. However, following the same ideas as in the proof of Theorem \ref{thm:traj_approx} in the Appendix, we can establish a connection between averaging theory and the results in this paper. 
We illustrate this idea using the introductory example. 
Consider \eqref{eq:intro1} and suppose that $f$ is continuously differentiable. After integrating of the differential equation, we obtain
\begin{equation}
x(t) = x_0 + \sqrt{\omega}\int_{t_0}^t \alpha \cos(\omega s) + f(x(s))\sin(\omega s)ds
\end{equation}
and by integrating the first expression of the integral $\int_{t_0}^t \alpha \cos(\omega s) = \frac{\alpha}{\omega}(\sin(\omega t)-\sin(\omega t_0))$ and performing a partial integration for the second expression $\int_{t_0}^t f(x(s))\sin(\omega s)ds = - \frac{1}{\omega}(f(x(t))\cos(\omega t)-f(x(t_0))\cos(\omega t_0))+\frac{1}{\omega}\int_{t_0}^t\nabla_x f(x(s))\dot{x}\cos(\omega s)ds$ we obtain
\begin{align}
 x(t) =&~ x_0+  \frac{\sqrt{\omega}}{\omega}r(\omega t,x(t))
 + \int_{t_0}^tb(\omega s,x(s)) ds
\end{align}
with 
\begin{align}
\nonumber \!\!\!\! r(\omega t,x(t)) =~& - f(x(t))\cos(\omega t)+f(x(t_0))\cos(\omega t_0)\\
&+\alpha(\sin(\omega t)-\sin(\omega t_0)) \\
\nonumber \!\!\!\! b(\omega s, x(s)) =~& \alpha\nabla_x f(x(s))\cos^2(\omega s)\\
&+ \nabla_x f(x(s)) f(x(s))\sin(\omega s)\cos(\omega s).
\end{align}
We see that for bounded trajectories the expression $\frac{\sqrt{\omega}}{\omega}r(\omega t,x(t))$ 
tends to zero when $\omega$ tends to infinity. Thus, we have that $x(t)\approx x_0+\int_{t_0}^t b(\omega s,x(s))ds$ and therefore $\dot{x} \approx  b(\omega t,x)$.
By rescaling time with $\tau=\omega t$ where $\omega = \frac{1}{\epsilon}$ we obtain
\begin{equation}
\begin{split}
\frac{dx}{d\tau} \approx \frac{1}{\omega} b(\tau,x) = \epsilon  b(\tau,x) 
\end{split}
\end{equation}
which is now in the form \eqref{eq:averaging}. We can use standard averaging analysis and obtain the averaged system
\begin{equation}
\begin{split}
\frac{dz}{d\tau} = \frac{1}{\omega} \frac{\alpha}{2}\nabla_z f(z) = \epsilon \frac{\alpha}{2}\nabla_z f(z)
\end{split}
\end{equation}
which coincides with \eqref{eq:intro2}. Summarizing, we established a connection between \eqref{eq:intro1} and \eqref{eq:intro2} using average-like arguments.


Notice that the amplitudes and frequencies of the sinusoids of the extremum seeking feedbacks in Fig. \ref{fig:singleint} and Fig. \ref{fig:unicycle} are different, compared to the amplitudes in the corresponding schemes in the existing literature \cite{zagsk}, \cite{zsk} and \cite{nesic}. Specifically, in \cite{zagsk} and \cite{zsk} the amplitudes of the perturbations are chosen to be $\omega$ and one, respectively, whereas the frequencies are chosen to be $\omega$. The choice of $\sqrt{\omega}$ for the amplitudes in combination with $\omega$ for the frequency is crucial in order to obtain the Lie bracket system \eqref{eq:liebracket_system} as approximation of the input-affine system \eqref{eq:input_affine_system} since the procedure described above would lead to a different averaged system for a different choice of the amplitudes. A similar remark was also pointed out on p. 241 in \cite{Kurzweil:1987fk}. 
Therefore, even though the schemes differ only in the choice of the amplitudes, the observation above let us expect that the average systems of the corresponding extremum seeking systems in \cite{zagsk} and \cite{zsk} differ from the Lie bracket systems obtained in this paper. A similar reasoning applies to \cite{nesic} concerning the results on static maps, where the parameters do not influence the frequencies of the perturbations but only their amplitudes. 
\subsection{Single-Agent Case}
Theorem \ref{thm:mas_single_int_loc} and Theorem \ref{thm:mas_single_int_glob} state local and semi-global practical uniform asymptotic stability for a group of $N$ agents with single-integrator and unicycle dynamics. A special case is a single-agent extremum seeking system for which we have $N=1$ and $f^1 = F$. Furthermore, a similar analysis can be adopted in a straight forward fashion to the case of extremum seeking in one dimension by removing one feedback loop in Fig. \ref{fig:singleint}. 
\subsection{Non-Sinusoidal Perturbations}
In the presented schemes in Fig. \ref{fig:singleintsimp}, Fig. \ref{fig:singleint} and Fig. \ref{fig:unicycle} it is not essential that the perturbation signals are sinusoidal. Theorem \ref{thm:liesystem_loc} and Theorem \ref{thm:liesystem_glob} can be applied to analogous schemes where the sinusoidal perturbations are replaced with other appropriately defined periodic signals as long as they satisfy Assumptions A3 and A4. This also includes discontinuous and/or non-differentiable signals such as square, triangle or sawtooth waveforms (see also Remark 4 above).

\section{Examples}
In this section, we show numerical examples which illustrate the main results. First, we compare for different values of $\omega$ the trajectories of the single-integrator system of \eqref{eq:multi_agent_input_affine_old} with its corresponding Lie bracket system \eqref{eq:liemultagentsingleint}. Second, using the Lie bracket system, we are able to explain characteristic points which are visible in the trajectories of the extremum seeking with unicycle dynamics. 

We consider a system of $N=3$ agents and enumerate them with $a,b,c$. We assign each agent the following maps
\begin{align}
\begin{split}
f^a(\bar{x}) =& -\frac{1}{2}(x_1^a-1)^2-\frac{1}{2}(x_2^a-1)^2 +{x_1^b}^2\\
&+{x_2^b}^2+e^{(-{x_1^c}^2-{x_2^c}^2)}-10,
\end{split} \\
\begin{split}
f^b(\bar{x}) =& -\frac{1}{2}(x_1^b+1)^2-\frac{1}{2}(x_2^b+1)^2\\
 &+ \sin(x_1^a+x_2^a)-10,
\end{split}\\
f^c(\bar{x}) =& -\frac{1}{2}(x_1^c+1)^2-\frac{3}{2}(x_2^c-1)^2+10.
\end{align}
We choose the parameters $h=h^a=h^b=h^c=1$, $\alpha^a=\alpha^b=\alpha^c=1$, $c^a=c^b=c^c=0.3$ and the initial conditions $[\bar{x}^\top_0,\bar{x}_{0e}^\top]^\top=[2,-2,-2,2,-1,2.5,0,0,0]^\top$. Observe that each of the $f^i$'s, $i=a,b,c$ are functions of the states of the respective other agents. 

Furthermore, we consider the quadratic function
\begin{equation}
F(\bar{x}) = -\frac{1}{2}(\bar{x}-\bar{x}^*)^\top Q (\bar{x}-\bar{x}^*),
\end{equation}
where $\bar{x}^*=[1, 1,-1, -1, -1, 1]^\top$ and the diagonal matrix $Q=\text{diag}(1,1,1,1,1,3)$. We can verify that $\nabla_{\bar{x}^i} f^i(\bar{x}) = \nabla_{\bar{x}^i}F(\bar{x})$, $i=a,b,c$ and we see that $F$ is quadratic and attains its maximal value at $\bar{x}^*$. We expect from Theorems \ref{thm:mas_single_int_glob} and \ref{thm:mas_unicycle_glob} that $\bigl[(\bar{x}^*)^{\top},\frac{f^a(\bar{x}^*)}{h},\frac{f^b(\bar{x}^*)}{h},\frac{f^c(\bar{x}^*)}{h})\bigr]^\top$ is semi-globally practically uniformly asymptotically stable for the extremum seeking systems.

In Fig. \ref{fig:singleint_10} the trajectories of the original and the Lie bracket systems are depicted with $\omega=10$ and $\omega^a=\omega$, $\omega^b=2\omega$, $\omega^c=3\omega$. The trajectories of the Lie bracket system captures the qualitative evolution of the trajectories of the original system. In Fig. \ref{fig:singleint_100} we see a simulation with the same parameters but with $\omega=100$. 

These examples illustrate two properties. First, the trajectories of the original system approach those of the Lie bracket system for large values of $\omega$. This observation points up the result of Theorem \ref{thm:traj_approx}. 
\begin{figure}[tpb]
\centering
\includegraphics[width=180pt]{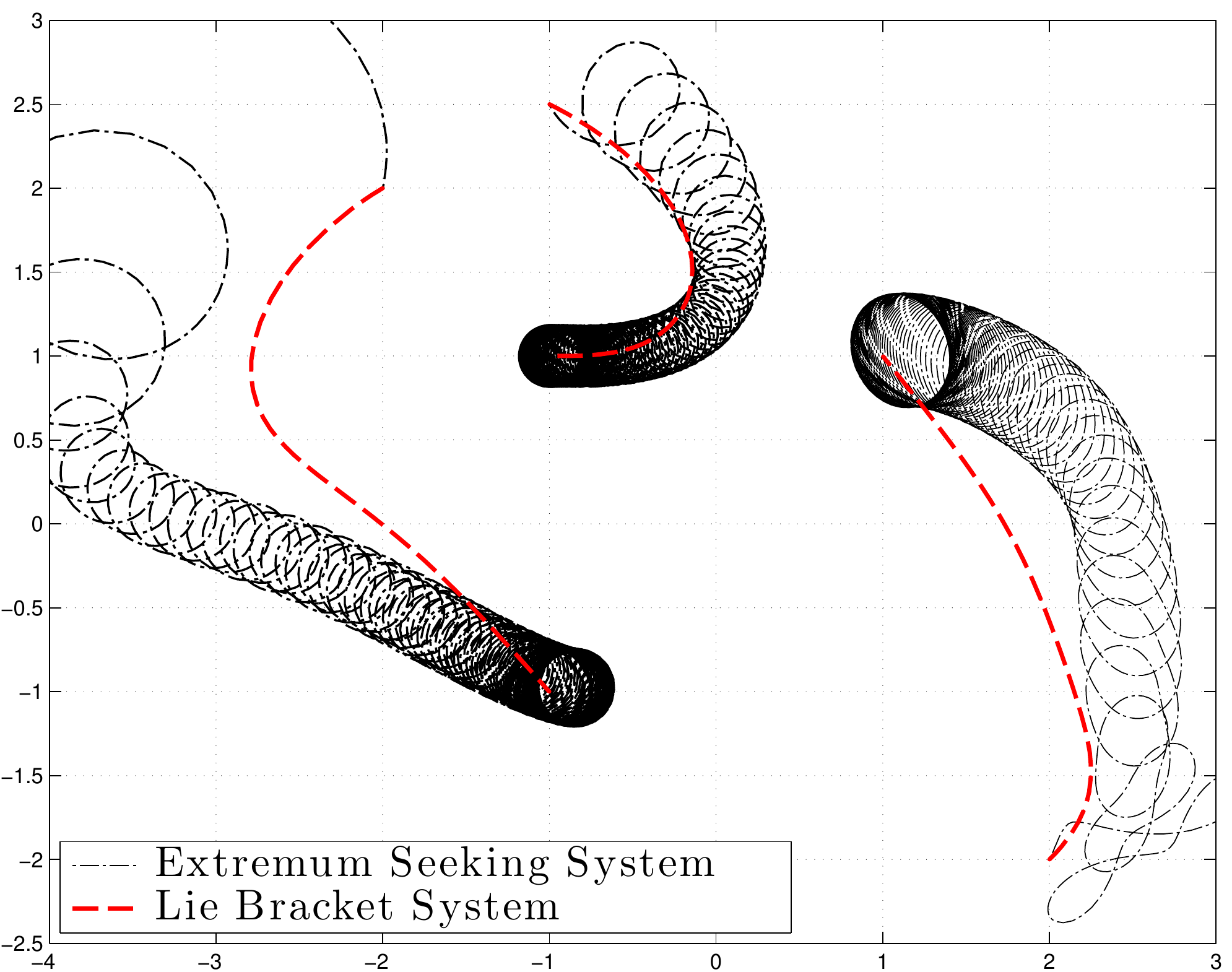}
\caption{Comparison of trajectories of a three-agent single-integrator system and its respective Lie bracket system, for $\omega = 10$}
\label{fig:singleint_10}
\end{figure}
Second, we deduce from Fig. \ref{fig:singleint_10} and Fig. \ref{fig:singleint_100} that even though each of the $f^i$'s, $i=a,b,c$ contains highly nonlinear terms depending on the states of the other agents, the overall system practically converges even for small values of $\omega$ to the expected extremum.
\begin{figure}[tpb]
\centering
\includegraphics[width=180pt]{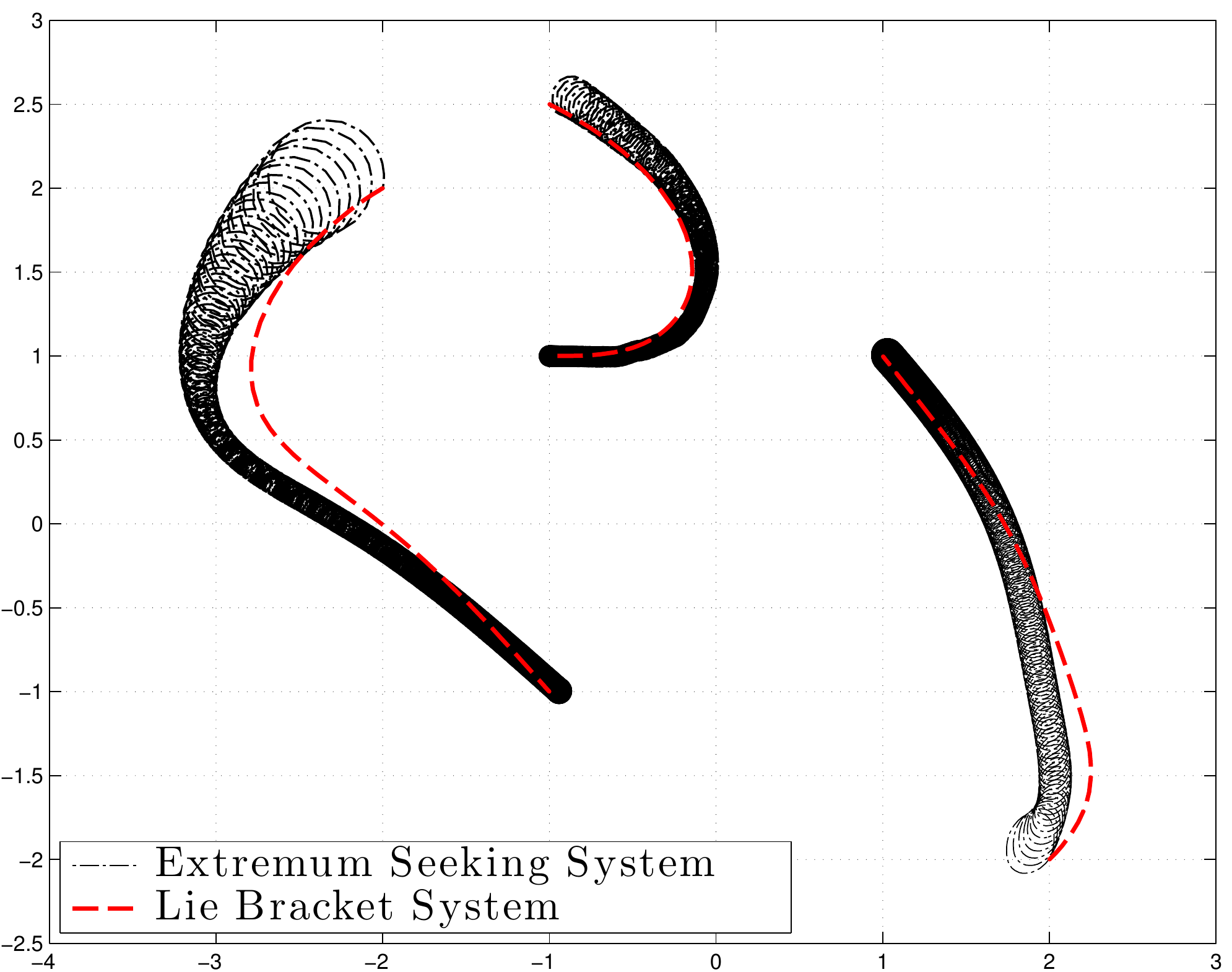}
\caption{Comparison of trajectories of a three-agent single-integrator system and its respective Lie bracket system, for $\omega=100$}
\label{fig:singleint_100}
\end{figure}

The same result can be observed in the case of unicycle dynamics and the same choice of parameters as above, with additionally $\Omega^a=1$, $\Omega^b = 2$, $\Omega^c=3$. In Fig. \ref{fig:unicycle100} the trajectories of the original and the Lie bracket systems are depicted for $\omega=80$. Observe that the overall system practically converges as expected to the extremum. The trajectory of the extremum seeking system contains characteristic points, which also appear in the trajectory of the Lie bracket system. Apparently the vector field changes its direction abruptly. This can be explained by regarding the differential equation of the Lie bracket system in \eqref{eq:lie_sys_unicycle}, which is time-varying and vanishes at the zero-crossing instances of the sinusoids. 
\begin{figure}[tpb]
\centering
\includegraphics[width=180pt]{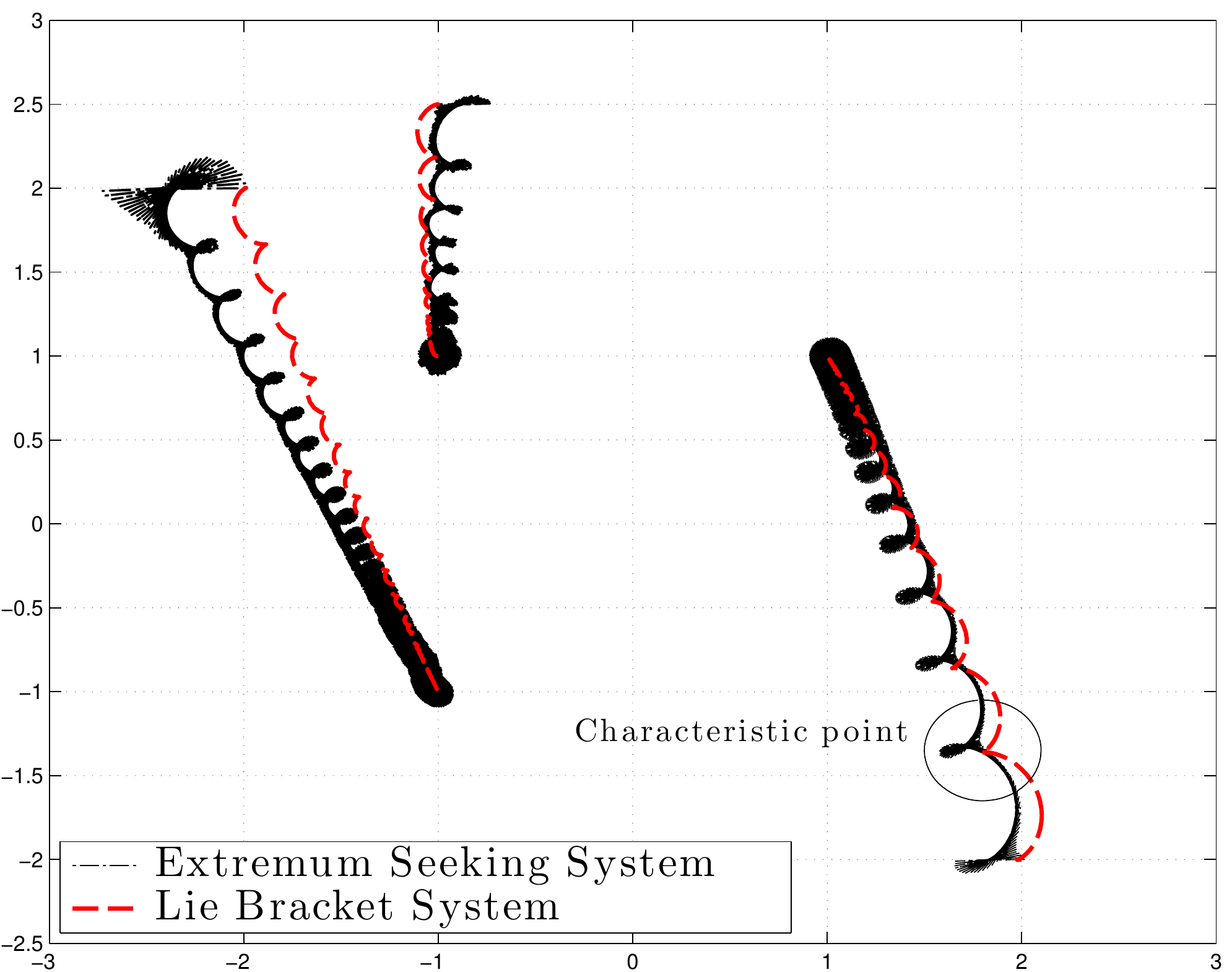}
\caption{Comparison of trajectories of a three-agent unicycle system and its respective Lie bracket system, for $\omega=80$}
\label{fig:unicycle100}
\end{figure}
\section{Conclusion}
In this work we developed a methodology, which led to a novel interpretation as well as to novel stability results for extremum seeking systems. By identifying the sinusoidal perturbations of the extremum seeking as artificial inputs, we were able to rewrite the system in a certain input-affine form and to relate this system to the so-called Lie bracket system, which nicely reveals the optimizing behavior of extremum seeking. The Lie bracket system viewpoint of extremum seeking allowed us to establish strong stability results for extremum seeking systems. We proved that the trajectories of systems belonging to a certain class of input-affine systems can be approximated by the trajectories of their corresponding Lie bracket system. Furthermore, we showed that global (local) uniform asymptotic stability of the Lie bracket system implies  semi-global (local) practical uniform asymptotic stability of the input-affine system. We applied these results to a multi-agent extremum seeking system consisting of agents with either single-integrator or unicycle dynamics. Finally, the results are illustrated using numerical examples.
\section{Acknowledgements}
We thank Shankar Sastry for the fruitful discussions and the anonymous referees for very helpful comments. This work was supported by the Deutsche Forschungsgemeinschaft (Emmy-Noether-Grant,  Novel Ways in Control and Computation, EB 425/2-1, and Cluster of Excellence in Simulation Technology, EXC 310/1), the Swedish Research Council and the Knut and Alice Wallenberg Foundation.
\begin{appendix}
\section{Existence and Uniqueness}
\label{app:exuni}
Consider the differential equation
\begin{equation}\label{eq:diffeq}
\dot{x}=f(t,x)
\end{equation}
with $f:\mathbb{R}\times\mathbb{R}^n\to\mathbb{R}^n$ and with initial condition $x(t_0)=x_0\in\mathbb{R}^n$. 
If there exist a $t_e\in(0,\infty)$ and an absolutely continuous function $x$ such that 
\begin{equation}
x(t) = x(t_0)+\int_{t_0}^{t} f(\tau,x(\tau))d\tau, ~~t\in[t_0,t_0+t_e)
\end{equation}
and $\dot{x}(t)=f(t,x(t))$ for $t\in[t_0,t_0+t_e)$ except on a set of measure zero, then $x(\cdot)=x(\cdot;t_0,x_0)$ is said to be a solution of \eqref{eq:diffeq} through $x(t_0) = x_0$ defined on $[t_0,t_0+t_e)$. 

\begin{thm}[see \cite{Bressan:2007fk,Hale:1969kx}]
\label{thm:exuni}
Consider \eqref{eq:diffeq} and suppose for each compact sets $\mathcal{T}\subseteq\mathbb{R}$ and $\mathcal{C}\subseteq\mathbb{R}^n$ there exist measurable functions $M,L:\mathcal{T}\to\mathbb{R}$ such that
\begin{equation}
\begin{split}
|f(t,x)|&\leq M(t), \\
|f(t,x_1)-f(t,x_2)|&\leq L(t)|x_1-x_2|, 
\end{split}
\end{equation}
$t\in \mathcal{T}, x,x_1,x_2\in \mathcal{C}$. 
Then for any $t_0\in \mathcal{T}$ and $x(t_0)\in \mathcal{C}$ there exist a $t_e\in(0,\infty)$ and a unique solution $x$ through $x(t_0)$, which is defined on $[t_0,t_0+t_e)$.
\end{thm}
\section{Preliminary Lemmas}
\label{appendix:proofThm1}
\begin{lemma}\label{lem:mas_simplification}
Let
\begin{equation}
\nu_{ij} =\frac{1}{2\pi} \int_{0}^{2\pi} u^i(n^i\tau)\int_{0}^{\tau}u^j(n^j \theta) d\theta d\tau
\end{equation}
with $n^i,n^j\in\mathbb{N}$, $u^i(n^it)\in \{\sin(n^it),\cos(n^it)\}$, then
\begin{equation}
\nu_{ij}=
\begin{cases}
\frac{1}{2 n^i} & 
\begin{split}
n^i=n^j, u^i(n^it) &= \sin(n^it),\\ u^j(n^jt) &= \cos(n^jt)
\end{split}\\
-\frac{1}{2 n^i} & 
\begin{split}
n^i=n^j, u^i(n^it) &= \cos(n^it), \\u^j(n^jt) &= \sin(n^jt)
\end{split}\\
0 & \text{ else }.
\end{cases}
\end{equation}
\end{lemma}
\begin{proof}
The result follows by a direct calculation.
\end{proof}
\begin{lemma}
\label{lem:average1}
Let $u:\mathbb{R}\times\mathbb{R}\to\mathbb{R}$ satisfy Assumption A\ref{ass:a3}. Furthermore, $u(t,\cdot)$ is $T$-periodic, i.e. $u(t,\theta+T)=u(t,\theta)$ for some $T\in(0,\infty)$ and all $t,\theta\in\mathbb{R}$. Then, there exist $k_1,k_2\in[0,\infty)$ such that the inequality
\begin{align}
\nonumber \biggl\lvert\int_{t_0}^{t}\biggl(u(\tau,\omega \tau) - \frac{1}{T}\int_{0}^{T}&u(\tau,\theta)d\theta\biggr) d\tau\biggr\lvert\\
\label{lemeq:1}
&\leq \frac{k_1(t-t_0)+k_2}{\omega}
\end{align}
is satisfied for all $t_0\in\mathbb{R}$ and all $t\in [t_0,\infty)$. Furthermore, $k_2=0$ if $\omega(t-t_0)$ is an integer multiple of $T$, i.e. there exists an $n\in\mathbb{N}_0$ such that $\omega(t-t_0)=Tn$.
\end{lemma}
\begin{proof}
Using the fact that $u(\tau,\omega\tau)$ $=\frac{1}{T}\int_0^Tu(\tau,\omega\tau)d\theta$ and applying the change of variables $r=\omega\tau$, $dr=\omega d\tau$, the expression in the norm of left hand-side in \eqref{lemeq:1} yields
\begin{equation}
\label{lemeq:2}
\begin{split}
\frac{1}{T\omega}\int_{0}^{T}\int_{\omega t_0}^{\omega t} u(\frac{r}{\omega},r) - u(\frac{r}{\omega},\theta) drd\theta.
\end{split}
\end{equation}
Since $T\in(0,\infty)$ we can divide $[\omega t_0,\omega t]$ into $n\in\mathbb{N}_{0}$ pieces of length $T$ such that $\omega (t-t_0) =Tn+\delta$ with $0\leq \delta < T$ being the leftover piece. We obtain for \eqref{lemeq:2}
\begin{align}
\nonumber&\; \frac{1}{T\omega}\sum_{k=0}^{n-1}\int_{0}^{T}\!\!\int_{\omega t_0+T k}^{\omega t_0+T(k+1)} u(\frac{r}{\omega},r) - u(\frac{r}{\omega},\theta) dr d\theta\\
\label{lemeq:3} &+ R_1,
\end{align}
where we introduced the left-over piece
\begin{equation}
\label{eq:defR1}
R_1:=\frac{1}{T\omega}\!\int_{0}^{T}\!\!\int_{\omega  t_0+T n}^{ \omega t_0+T n+\delta} \!\!u(\frac{r}{\omega},r) - u(\frac{r}{\omega},\theta) drd\theta,
\end{equation}
which is considered later. 

The integration interval in \eqref{lemeq:3} is now shifted by introducing the change of variable $s=r-\omega t_0-Tk$, $ds=dr$
\begin{align}
\nonumber&\frac{1}{T\omega}\sum_{k=0}^{n-1}\int_{0}^{T}\int_{0}^{T} u(\frac{h_k(s)}{\omega},h_k(s))- u(\frac{h_k(s)}{\omega},\theta) dsd\theta\\
&+ R_1
\end{align} 
with $h_k(s):=s+\omega t_0+T k$.
Since $u(t,\cdot)$ is $T$-periodic, it follows that $u(\frac{h_k(s)}{\omega},h_k(s))=u(\frac{h_k(s)}{\omega},h_0(s))$. Thus, this simplifies to
\begin{align}
\nonumber&\frac{1}{T\omega}\sum_{k=0}^{n-1}\int_{0}^{T}\int_{0}^{T}u(\frac{h_k(s)}{\omega},h_0(s))- u(\frac{h_k(s)}{\omega},\theta)ds d\theta\\
&+ R_1.
\end{align}
Note, that since the integration with respect to $s$ and with respect to $\theta$ is performed from $0$ to $T$ and due to the periodicity of $u(t,\cdot)$, we can add $\int_{0}^Tu(\frac{h_k(0)}{\omega},\theta)d\theta$ and subtract $\int_{0}^Tu(\frac{h_k(0)}{\omega},h_0(s))ds$ which sums up to zero. We obtain
\begin{align}
\nonumber&\frac{1}{T\omega}\sum_{k=0}^{n-1}\int_{0}^{T}\int_{0}^{T}u(\frac{h_k(s)}{\omega},h_0(s))-u(\frac{h_k(0)}{\omega},h_0(s))\\
\label{lemeq:4}&+u(\frac{h_k(0)}{\omega},\theta)- u(\frac{h_k(s)}{\omega},\theta) dsd\theta+ R_1.
\end{align}
Assumption A\ref{ass:a3} yields the existence of $L\in(0,\infty)$ such that the above expression can be bounded from above as follows $|u(\frac{h_k(s)}{\omega},h_0(s))-u(\frac{h_k(0)}{\omega},h_0(s))|\leq \frac{L}{\omega}|s|$ and $|u(\frac{h_k(0)}{\omega},\theta)- u(\frac{h_k(s)}{\omega},\theta)|\leq \frac{L}{\omega}|s|$. Thus, \eqref{lemeq:4} can be upper bounded by
\begin{align}
\nonumber&\frac{1}{T\omega}\sum_{k=0}^{n-1}\int_{0}^{T}\int_{0}^{T}2\frac{L}{\omega}|s|  dsd\theta+ |R_1|\\
&= \frac{T^2 L}{\omega^2}n + |R_1|.
\end{align} 
We now consider the expression $R_1$ in \eqref{eq:defR1}. Assumption A\ref{ass:a3} yields the existence of $M\in(0,\infty)$ such that it can be upper bounded as follows
\begin{equation}
\begin{split}
|R_1| 
&\leq\frac{1}{T\omega}\int_{0}^{T}\int_{\omega t_0+ T n}^{\omega t_0+ T n+\delta} 2M  d\tau d\theta = \frac{2M\delta}{\omega}.
\end{split}
\end{equation}
Therefore, using the definition of $n=\frac{\omega (t-t_0)-\delta}{T}$ we obtain
\begin{equation}
\begin{split}
\frac{T^2 L}{\omega^2}n + \frac{2M\delta}{\omega} &=\frac{T^2 L}{\omega^2}\frac{\omega (t-t_0)-\delta}{T} +\frac{2M\delta}{\omega}\\
&\leq \frac{T L(t-t_0)+2M\delta}{\omega}.
\end{split}
\end{equation}
Choosing $k_1:=T L$ and $k_2 := 2M\delta$ proves the first claim. If $\omega(t-t_0)=Tn$ then $\delta=0$ and therefore, $k_2=0$ which proves the second claim.
\end{proof}

\begin{lemma}
\label{lem:average2}
Let $u_i,u_j:\mathbb{R}\times\mathbb{R}\to\mathbb{R}$ satisfy Assumptions A\ref{ass:a3} and A\ref{ass:a4}. Furthermore, let 
\begin{equation}
\tilde{u}_{ij}(t,\theta):=u_i(t,\theta)\int_0^\theta u_j(t,r)dr,
\label{eq:uijtilde}
\end{equation}
then there exist $M_{ij},L_{ij}\in(0,\infty)$ such that
\begin{enumerate}
\item $\tilde{u}_{ij}(t,\cdot)$ is $T$-periodic, i.e. $\tilde{u}_{ij}(t,\theta + T) = \tilde{u}_{ij}(t,\theta)$,
\item $\sup_{t,\theta\in\mathbb{R}}|\tilde{u}_{ij}(t,\theta)|\leq M_{ij}$,
\item $|\tilde{u}_{ij}(t_1,\theta)-\tilde{u}_{ij}(t_2,\theta)| \leq L_{ij}|t_1-t_2|$.
\end{enumerate}
\end{lemma}
\begin{proof}
To (1): Consider $\tilde{u}_{ij}(t,\theta + T)$.
Performing a change of variables $s=r-T$ and $ds=dr$ yields
\begin{equation}
\begin{split}
u_i&(t,\theta+T)\int_0^{\theta+T}\!\!\! u_j(t,r)dr\\
&= u_i(t,\theta)\int_{-T}^{\theta} u_j(t,s+T)ds,
\end{split}
\end{equation}
where we made use of $T$-periodicity of $u_i(t,\cdot)$ in Assumption A\ref{ass:a4}. Again, due to Assumption A\ref{ass:a4} $u_j(t,\cdot)$ has zero average and is $T$-perodic. Thus, the expression above yields 
\begin{equation}
\begin{split}
\underbrace{u_i(t,\theta)\int_{-T}^{0} u_j(t,s)ds}_{=0}+u_i(t,\theta)\int_0^{\theta} u_j(t,r)dr.
\end{split}
\end{equation}

To (2): Since $T\in(0,\infty)$ we can divide $[0,\theta]$ into $n\in\mathbb{N}_0$ pieces of length $T$ such that $\theta =Tn+\delta$ with $0\leq \delta < T$ being the leftover piece. Due to Assumption A\ref{ass:a4}, the first pieces are zero. Thus, we obtain
\begin{equation}
\begin{split}
|\tilde{u}_{ij}(t,\theta)| =&\; |\underbrace{u_i(t,\theta)\sum_{k=0}^{n-1}\int_{kT}^{(k+1)T} u_j(t,r)dr}_{=0}\\
&+u_i(t,\theta)\int_{nT}^{nT+\delta} u_j(t,r)dr|\\
\leq&\; M_iM_j(\theta-nT) \leq \underbrace{M_iM_jT}_{ =:M_{ij}},
\end{split}
\end{equation}
where the last step follows from Assumption A\ref{ass:a3}.

To (3): Using the definition of $\tilde{u}_{ij}$ in \eqref{eq:uijtilde} we can add and subtract the term $u_i(t_1,\theta)\int_0^{\theta} u_j(t_2,r)dr$ which yields
\begin{equation}\label{lemeq:2nd31}
\begin{split}
&|\tilde{u}_{ij}(t_1,\theta)-\tilde{u}_{ij}(t_2,\theta)|\\
=&\; |u_i(t_1,\theta)\int_0^{\theta} (u_j(t_1,r)-u_j(t_2,r)dr)\\
&\;+(u_i(t_1,\theta)-u_i(t_2,\theta))\int_0^{\theta} u_j(t_2,r)dr|. \\
\end{split}
\end{equation}
Since $T\in(0,\infty)$ we can divide $[0,\theta]$ into $n\in\mathbb{N}_0$ pieces of length $T$ such that $\theta =Tn+\delta$ with $0\leq \delta < T$ being the leftover piece. We obtain for the expression above
\begin{align}
\nonumber=&\; |u_i(t_1,\theta)\sum_{k=0}^{n-1}\int_{kT}^{(k+1)T} (u_j(t_1,r)-u_j(t_2,r)dr)\\
\nonumber&\;+u_i(t_1,\theta)\int_{nT}^{nT+\delta} (u_j(t_1,r)-u_j(t_2,r)dr)\\
\nonumber&\;+(u_i(t_1,\theta)-u_i(t_2,\theta))\sum_{k=0}^{n-1}\int_{kT}^{(k+1)T} u_j(t_2,r)dr\\
\label{lemeq:5}
&\;+(u_i(t_1,\theta)-u_i(t_2,\theta))\int_{nT}^{nT+\delta} u_j(t_2,r)dr|. 
\end{align}
The first and third line in \eqref{lemeq:5} sum up to zero due to Assumption A\ref{ass:a4}. Furthermore, due to Assumptions A\ref{ass:a3} we obtain
\begin{equation}\label{lemeq:2nd33}
\begin{split}
\leq &| u_i(t_1,\theta)|\int_{nT}^{nT+\delta} L_j|t_1-t_2|dr\\
&+L_i|t_1-t_2|\int_{nT}^{nT+\delta} |u_j(t_2,r)|dr\\
\leq &(M_iL_j+L_iM_j)\delta|t_1-t_2| \\
\leq& \underbrace{(M_iL_j+L_iM_j)T}_{=:L_{ij}}|t_1-t_2|.
\end{split}
\end{equation}
This was the last property we had to prove.
\end{proof}

\begin{lemma} 
\label{lem:average3}
Let $u_i, u_j:\mathbb{R}\times\mathbb{R}\to\mathbb{R}$ satisfy Assumptions A\ref{ass:a3} and A\ref{ass:a4}. Then there exist $k_1,k_2,k_3,k_4\in[0,\infty)$ such that the following inequality
\begin{align}
\nonumber&\biggl|\int_{t_0}^{t}\biggl(\omega u_i(\tau, \omega \tau)\int_{t_0}^{\tau}
u_j(s,\omega s)ds\\
\label{lemeq:6}&-\frac{1}{T} \int_0^{T}\biggl[u_i(\tau,\theta)\int_{0}^{\theta} u_j(\tau, r)dr\biggr]d\theta\biggr) d\tau\biggr| \\
\nonumber&\leq k_1\frac{(t-t_0)^2}{\omega}+k_2\frac{t-t_0}{\omega} +k_3\frac{1}{\omega}+k_4\frac{1}{\omega^2}+k_5\frac{1}{\omega^3}
\end{align}
is satisfied for all $t_0\in\mathbb{R}$ and all $t\in [t_0,\infty)$.
\end{lemma}
\begin{proof}
In order to use Lemma \ref{lem:average1} we add and subtract $\int_{t_0}^{t}\tilde{u}_{ij}(\tau, \omega \tau)d\tau=$  $\int_{t_0}^{t}( u_i(\tau, \omega \tau)\int_{0}^{\omega\tau}u_j(\tau,r)dr)d\tau$ (see \eqref{eq:uijtilde}) in the norm on the left hand-side of \eqref{lemeq:6}. Thus, it can be written as 
\begin{align}
\label{lemeq:10} \int_{t_0}^{t}\biggl( \tilde{u}_{ij}(\tau, \omega \tau)-\frac{1}{T} \int_0^{T}\tilde{u}_{ij}(\tau,\theta)d\theta\biggr) d\tau + R
\end{align}
with
\begin{equation}
\begin{split}
\label{lemeq:defR}
R:=\int_{t_0}^{t}&\biggl(\omega u_i(\tau, \omega \tau)\int_{ t_0}^{\tau} u_j(s,\omega s)ds\\
&- \tilde{u}_{ij}(\tau, \omega \tau)\biggr)d\tau.
\end{split}
\end{equation}
Due to Lemma \ref{lem:average2} the expression $\tilde{u}_{ij}$ in \eqref{lemeq:10} satisfies all assumptions needed in Lemma \ref{lem:average1} which can now be applied in order to establish the existence of $\tilde{k}_1,\tilde{k}_2\in[0,\infty)$ such that 
\begin{equation}
\begin{split}
\label{eq:upbound0}
\biggl|\int_{t_0}^{t}\biggl(\tilde{u}_{ij}\left(\tau,\omega\tau\right) -\frac{1}{T} \int_0^{
T}&\tilde{u}_{ij}\left(\tau,\theta\right) d\theta \biggr)d\tau\biggr|\\
&\leq \frac{\tilde{k}_1(t-t_0)+\tilde{k}_2}{\omega}.
\end{split}
\end{equation}
In the following we establish an upper bound for $R$. We first split up the integration interval in \eqref{lemeq:defR}, i.e. $\int_{0}^{\omega\tau}u_j(\tau,r)dr = \int_{0}^{\omega t_0}u_j(\tau,r)dr+\int_{\omega t_0}^{\omega\tau}u_j(\tau,r)dr$ and obtain
\begin{equation}
\begin{split}
R=&\int_{t_0}^{t}\biggl(\omega u_i(\tau, \omega \tau)\biggl[\int_{ t_0}^{\tau} u_j(s,\omega s)ds \\
&-\frac{1}{\omega}\int_{\omega t_0}^{\omega\tau}\!\!\!\!u_j(\tau,r)dr\biggr]\biggr)d\tau+R_1,
\end{split}
\end{equation}
where we introduced
\begin{equation}
\label{lemeq:defR1}
R_1:=-\int_{t_0}^{t}\biggl(u_i(\tau, \omega \tau)\int_{0}^{\omega t_0}\!\!\!\!u_j(\tau,r)dr\biggr)d\tau.
\end{equation}
By the changes of variables $p=\omega \tau$, $dp=\omega d\tau$ and $q=\omega s$, $dq=\omega ds$ we obtain
\begin{equation}
\begin{split}
R=&\frac{1}{\omega}\int_{\omega t_0}^{\omega t}\biggl(u_i(\frac{p}{\omega}, p)\biggl[\int_{\omega t_0}^{p} u_j(\frac{q}{\omega}, q)dq \\
&-\int_{\omega t_0}^{p}\!\!\!\!u_j(\frac{p}{\omega},r)dr\biggr]\biggr)dp +R_1.
\end{split}
\end{equation}
Since the integration intervals with respect to $r$ and $q$ are now equal, we combine the two inner integrals and introduce $I(q,p):=u_j(\frac{q}{\omega}, q)-u_j(\frac{p}{\omega},q)$. Furthermore, we divide $[\omega t_0,\omega t]$ into $n\in\mathbb{N}_{0}$ pieces of length $T$ such that $\omega (t-t_0) =Tn+\delta$ with $0\leq \delta < T$ being the leftover piece. Thus, we have
\begin{align}
\nonumber R=&\frac{1}{\omega}\sum_{k=0}^{n-1}\int_{\omega t_0+Tk}^{\omega t_0+ T(k+1)} \biggr[u_i(\frac{p}{\omega}, p)\int_{ \omega t_0}^{p} I(q,p)dq\biggl]dp\\
\nonumber&+\frac{1}{\omega}\int_{\omega t_0+Tn}^{\omega t_0+ Tn+\delta}\biggl[u_i(\frac{p}{\omega}, p)\int_{ \omega t_0}^{p}I(q,p)dq\biggr]dp\\
&+R_1.
\end{align}
For reasons which become clear later, we again split up the integration interval $\int_{ \omega t_0}^{p} I(q,p)dq = \int_{ \omega t_0}^{\omega t_0+Tk} I(q,p)dq+\int_{ \omega t_0+Tk}^{p} I(q,p)dq$, $k=1,\ldots,n$ and obtain
 $
R=R_1+R_2+R_3
$, 
where we define
\begin{align}
\nonumber \!\!\!\!\!R_2:=&\frac{1}{\omega}\sum_{k=0}^{n-1}\int_{\omega t_0+Tk}^{\omega t_0+ T(k+1)}\biggr[u_i(\frac{p}{\omega}, p)\!\!\int_{ \omega t_0}^{\omega t_0+Tk}\!\!\!\!\!\!\!\!\!\!\!\!\!\!\!\!I(q,p)dq\biggr]dp \\
&+\frac{1}{\omega}\int_{\omega t_0+Tn}^{\omega t_0+ Tn+\delta}\biggl[u_i(\frac{p}{\omega}, p)\!\!\int_{ \omega t_0}^{\omega t_0+Tn}\!\!\!\!\!\!\!\!\!\!\!\!\!\!\!\!  I(q,p)dq\biggr]dp
\end{align}
and 
\begin{align}
\nonumber\!\!\!\!\!R_3:=&\frac{1}{\omega}\sum_{k=0}^{n-1}\int_{\omega t_0+Tk}^{\omega t_0+ T(k+1)}\biggl[u_i(\frac{p}{\omega}, p)\!\!\int_{ \omega t_0+Tk}^{p}\!\!\!\!\!\!\!\!\!\!\!\!\!\!\!\! I(q,p)dq\biggr]dp\\
&+\frac{1}{\omega}\int_{\omega t_0+Tn}^{\omega t_0+ Tn+\delta}\biggl[u_i(\frac{p}{\omega}, p)\!\!\int_{ \omega t_0+Tn}^{p}\!\!\!\!\!\!\!\!\!\!\!\!\!\!\!\! I(q,p)dq\biggr]dp.
\end{align}
Each part is now treated separately. 

For $R_1$ we split up the second integration interval $[0,\omega t_0]$ by introducing $l\in\mathbb{N}_0$ such that $\omega t_0=Tl+\epsilon$ with $0\leq \epsilon < T$ being the left-over piece. We know from Assumption A\ref{ass:a4} that $u_j(t,\cdot)$ has zero average.  Thus \eqref{lemeq:defR1} simplifies to
\begin{equation}
\begin{split}
R_1&=-\int_{t_0}^{t}\biggl(u_i\left(\tau,\omega \tau\right)\int_{Tl}^{Tl+\epsilon}u_j\left(\tau, s\right)ds\biggr)d\tau\\
\end{split}
\end{equation}
and with Assumption A\ref{ass:a3}, i.e. $|\int_{Tl}^{Tl+\epsilon}u_j\left(\tau, s\right)ds| \leq M_j\epsilon$ the expression $\bar{u}_{ij}(\tau,\theta):=u_i(\tau,\theta)\int_{Tl}^{Tl+\epsilon}u_j(\tau, s)ds$ is bounded, $\bar{u}_{ij}(\tau,\cdot)$ is $T$-periodic with zero mean and $\bar{u}_{ij}(\cdot,\theta)$ is Lipschitz continuous which follows from the same reasoning as in the proof of part (3) of Lemma \ref{lem:average2} (i.e. \eqref{lemeq:2nd31} to \eqref{lemeq:2nd33}). Thus it satisfies all assumptions of Lemma \ref{lem:average1}. We conclude with the first statement of Lemma \ref{lem:average1} that there exist $\tilde{k}_3,\tilde{k}_4\in[0,\infty)$ such that 
\begin{equation}
\label{eq:upboundR1}
|R_1|\leq\frac{\tilde{k}_3(t-t_0)+\tilde{k}_4}{\omega}.
\end{equation}

We now turn to $R_2$. Since $u_j(t,\cdot)$ is $T$-periodic with zero mean, we have that $\int_{\omega t_0}^{\omega t_0+Tk}u_j(\frac{p}{\omega},q)dq=0$ and therefore $\int_{ \omega t_0}^{\omega t_0+Tk} I(q,p)dq = \int_{\omega t_0}^{\omega t_0+T k}u_j\left(\frac{q}{\omega}, q\right)dq$ $k=1,\ldots,n$. The crucial point now is that this integral does not depend on $p$ anymore. Thus the expression $R_2$ can be written as
\begin{align}
\nonumber \!\!\!\!R_2=&\frac{1}{\omega} \sum_{k=0}^{n-1} \int_{\omega t_0+Tk}^{\omega t_0+T(k+1)}\!\!\!\!\!\!\!\! u_i(\frac{p}{\omega},p)dp\int_{\omega t_0}^{\omega t_0+T k}\!\!\!\!\!\!\!\!u_j(\frac{q}{\omega}, q)dq \\ 
&+\frac{1}{\omega}\int_{\omega t_0+Tn}^{\omega t_0+ Tn+\delta}\!\!\!\!\!\!\!\! u_i(\frac{p}{\omega}, p)dp\int_{ \omega t_0}^{\omega t_0+Tn}\!\!\!\!\!\!\!\!u_j(\frac{q}{\omega}, q)dq.
\end{align}
  Substituting $r=\frac{p}{\omega}$, $dr=\frac{dp}{\omega}$ and $s=\frac{q}{\omega}$,  $ds=\frac{dq}{\omega}$ yields
\begin{align}
\nonumber \!\!\!\! R_2=&\;\omega\sum_{k=0}^{n-1} \int_{ t_0+\frac{Tk}{\omega}}^{ t_0+\frac{T(k+1)}{\omega}}\!\!\!\!\!\!\!\! u_i\left(r,\omega r\right)dr\int_{t_0}^{t_0+\frac{T k}{\omega}}\!\!\!\!\!\!u_j\left(s, \omega s\right)ds\\
\label{eq:R2integrands}&+\omega\int_{ t_0+\frac{Tn}{\omega}}^{ t_0+\frac{Tn+\delta}{\omega}}\!\!\!\!\!\!\!\! u_i\left(r,\omega r\right)dr\int_{t_0}^{t_0+\frac{T n}{\omega}}\!\!\!\!\!\!u_j\left(s, \omega s\right)ds.
\end{align}
We now treat each integral in \eqref{eq:R2integrands} separately. Since $u_i$ is bounded by $M_i\in(0,\infty)$ we can upper bound the third integral and since both $u_i,u_j$ satisfy the conditions of Lemma \ref{lem:average1} and $\omega(t_0+\frac{T(k+1)}{\omega}-t_0-\frac{Tk}{\omega})=T$ as well as $\omega(t_0+\frac{Tk}{\omega}-t_0)=Tk$, $k=1,\ldots,n$, we obtain for the first, second and fourth integral with the second statement of Lemma \ref{lem:average1} that there exist $\tilde{k}_5,\tilde{k}_6,\tilde{k}_7\in[0,\infty)$ such that
\begin{align}
\nonumber\!\!\!\!\!|R_2|&\leq \omega\sum_{k=0}^{n-1}\frac{\tilde{k}_5 T}{\omega^2}\frac{\tilde{k}_6 Tk}{\omega^2}+M_i\delta\frac{\tilde{k}_7Tn}{\omega^2}\\
\label{eq:upboundR2}&\leq \tilde{k}_5\tilde{k}_6\left(\frac{(t-t_0)^2}{\omega}+\frac{\delta^2}{\omega^3}\right)+\frac{\tilde{k}_7M_iT(t-t_0)}{\omega},
\end{align}
where we have made use of $0\leq\delta<T$ and the definition of $n=\frac{\omega (t-t_0)-\delta}{T}$ above. 

For $R_3$ we proceed as follows. Note that due to Assumption A\ref{ass:a3} we have that $|I(q,p)|\leq \frac{L_j}{\omega}|q-p|$ and furthermore, $|u_i(t,\theta)|\leq M_i$, for all $t,\theta\in\mathbb{R}$ and all $i,j=1,\ldots,m$.  Thus, we obtain for $|R_3|$
\begin{align}
\nonumber\!\!\!\!|R_3|\leq &\;\frac{1}{\omega} \sum_{k=0}^{n-1} \int_{\omega t_0+Tk}^{\omega t_0+T(k+1)}\!\!\!\!\!\!\!\!\!\! M_i\int_{\omega t_0+Tk}^{p}\frac{L_j}{\omega}|q-p|dq dp \\
&+\;\frac{1}{\omega} \int_{\omega t_0+Tn}^{
\omega t_0+Tn+\delta}\!\!\!\!\!\!\!\!\!\! M_i\int_{\omega t_0+Tn}^{p}\frac{L_j}{\omega}|q-p|dq dp.
\end{align}
The crucial point now is that the lower integration limits of both integrations are equal. One can verify that after the substitutions $s=q-\omega t_0-Tk$, $ds=dq$ and $r=p-\omega t_0-Tk$, $dr=dp$, $k=1,\ldots,n$ we obtain
\begin{align}
\nonumber|R_3|\leq &\;\frac{M_iL_j}{\omega^2} \Biggl(\sum_{k=0}^{n-1} \int_0^{
T}\!\!\! \int_{0}^r\!\! |s-r|dsdr+\!  \int_0^{\delta}\!\!\! \int_{0}^r\!\! |s-r|dsdr\Biggr) \\ 
=&\frac{M_iL_j}{6\omega^2}\left(T^3n+\delta^3\right).
\end{align}
Using the definition of $n=\frac{\omega (t-t_0)-\delta}{T}$ above, we obtain
\begin{equation}
\label{eq:upboundR3}
\begin{split}
|R_3|\leq \frac{T^2M_iL_j(t-t_0)}{6\omega}+\frac{M_iL_jT^3}{6\omega^2},
\end{split}
\end{equation}
where we have used $0<\delta\leq T$. With $R=R_1+R_2+R_3$ in \eqref{lemeq:defR}, \eqref{eq:upbound0}, \eqref{eq:upboundR1}, \eqref{eq:upboundR2} and \eqref{eq:upboundR3} we obtain the desired upper bound for the left hand-side of \eqref{lemeq:6} with $k_1=\tilde{k}_5\tilde{k}_6$, $k_2=\tilde{k}_1+\tilde{k}_3+\frac{T^2M_iL_j}{6}+\tilde{k}_7M_iT$, $k_3=\tilde{k}_2+\tilde{k}_4$, $k_4 = \frac{M_iL_jT^3}{6\omega^2}$ and $k_5=\tilde{k}_5\tilde{k}_6T^2$.
\end{proof}
\section{Proof of Theorem \ref{thm:traj_approx}}
\label{pf:traj_approx}

Consider the vector field $f_\omega(t,x)=b_0(t,x)$   $+\sum_{i=1}^{m}b_{i}(t,x)\sqrt{\omega}u_{i}(t,\omega t)$ in \eqref{eq:input_affine_system} and note that due to Assumptions A\ref{ass:a1} and A\ref{ass:a3} $f_\omega(t,\cdot)$ is continuously differentiable and $f_\omega(\cdot,x)$ is measurable. Furthermore, with Assumption A\ref{ass:a2} we have that for every compact set $\mathcal{C}\subseteq\mathbb{R}^n$ and every $\omega\in(0,\infty)$ there exist $M,L\in[0,\infty)$ such that $|b_0(t,x)+\sum_{i=1}^{m}b_{i}(t,x)\sqrt{\omega}u_{i}(t,\omega t)|\leq M$ and such that $|b_0(t,x_1)+\sum_{i=1}^{m}b_{i}(t,x_1)\sqrt{\omega}u_{i}(t,\omega t) - b_0(t,x_2)-\sum_{i=1}^{m}b_{i}(t,x_2)\sqrt{\omega}u_{i}(t,\omega t)|\leq |b_0(t,x_1)-b_0(t,x_2)|+\sqrt{\omega}\sum_{i=1}^{m}M_i|(b_{i}(t,x_1)-b_{i}(t,x_2))|\leq L|x_1-x_2|$, $t\in\mathbb{R}, x, x_1,x_2\in\mathcal{C}$. We conclude with Theorem \ref{thm:exuni} in Appendix \ref{app:exuni}, for every $\omega\in(0,\infty)$, every $t_0\in\mathbb{R}$ and every $x_0\in \mathbb{R}^n$ there exist a $t_e\in(0,\infty)$ and a unique absolutely continuous solution $x$ of \eqref{eq:input_affine_system} such that
\begin{equation}\label{eq:C1}
\begin{split}
\!\!\!\!\!x(t) = x_0\! +\!\! \int_{t_0}^{t}b_0(\tau,x)\! +\! \sum_{i=1}^m b_i(\tau,x)\sqrt{\omega}u_i(\tau,\omega\tau) d\tau\!\!\!
\end{split}
\end{equation}
with $t\in[t_0,t_0+t_e)$ and $x_0=x(t_0)$. Since, $x(t)$ is absolutely continuous on $[t_0,t_0+t_e)$ we can perform a partial integration (see Thm. 4 on p. 266 in \cite{Nathanson:1964kx}) for each $b_i(\tau,x)\sqrt{\omega}u_i(\tau,\omega\tau)$, $i=1,\ldots,m$ with derivative $\frac{db_i(\tau,x)}{d\tau}=\frac{\partial b_i(\tau,x)}{\partial x}\dot{x}+\frac{\partial b_i(\tau,x)}{\partial \tau}$ almost everywhere and obtain
\begin{align}
\nonumber\!\!\!\!\!\!x(t)=&\;x_0\!+\!\int_{t_0}^{t}\biggl[b_0(\tau,x)\\
\nonumber &-\! \sqrt{\omega}\sum_{i=1}^m \biggl(\frac{\partial b_i(\tau,x)}{\partial x}\dot{x}+\frac{\partial b_i(\tau,x)}{\partial \tau}\biggr)U_i(t_0,\tau) \biggr]d\tau\\ \!\!\!\!
 &+\! \sqrt{\omega} \sum_{i=1}^m b_i(t,x(t))U_i(t_0,t)
\end{align}
with $U_i(t_0,t):=\int_{t_0}^tu_i(r,\omega r)dr$.
Since $\dot{x}(t)=b_0(t,x(t))+\sum_{i=1}^{m}b_{i}(t,x(t))\sqrt{\omega}u_{i}(t,\omega t)$ for almost all $t$, we obtain
\begin{align}
\nonumber x(t)=&\;x_0+\int_{t_0}^{t}\biggl[b_0(\tau,x)\!\\
\nonumber &- \omega\sum_{\substack{i,j=1}}^m \!\frac{\partial b_i(\tau,x)}{\partial x} b_j(\tau,x)u_j(\tau,\omega\tau)U_i(t_0,\tau)\biggr] d\tau\!\\
&+R_1+R_2,
\end{align}
where we introduced
\begin{align}
\nonumber R_1:=-&\sqrt{\omega}\int_{t_0}^{t}\biggl[\sum_{i=1}^m \biggl(\frac{\partial b_i(\tau,x)}{\partial x} b_0(\tau, x) \\
&+\sum_{i=1}^m \frac{\partial b_i(\tau,x)}{\partial\tau}\biggr)U_i(t_0,\tau)\biggr] d\tau\\
R_2:=& \sqrt{\omega} \sum_{i=1}^m b_i(t,x(t))U_i(t_0,t).
\end{align}
Adding and subtracting the expression 
$\omega\int_{t_0}^{t} \sum_{i=1}^m\sum_{j=i+1}^{m} \frac{\partial b_j(\tau,x)}{\partial x}
b_i(\tau,x)u_j(\tau,\omega\tau)U_i(t_0,\tau)d\tau$ yields
\begin{align}\label{eq:C6}
\nonumber x(t)=&\;x_0\!+ \int_{t_0}^{t}\biggl[b_0(\tau,x)\\
\nonumber&+ \omega \sum_{\substack{i=1\\j=i+1}}^m [b_i,b_j](\tau,x) u_j(\tau,\omega\tau)U_i(t_0,\tau)\biggr]d\tau\\
&+R_1+R_2+R_3+R_4
\end{align}
with
\begin{align}
\!\!\!\!R_3:=&-\omega\!\!\int_{t_0}^{t} \sum_{i=1}^m\frac{\partial b_i(\tau,x)}{\partial x} b_i(\tau,x)\frac{1}{2}\frac{\partial U_i(t_0,\tau)^2}{\partial \tau}d\tau\!\!\!\\
\nonumber \!\!\!\!R_4:=&-\omega\!\!\int_{t_0}^{t}\sum_{i=1}^m\sum_{j=1}^{i-1}  \frac{\partial b_i(\tau,x)}{\partial x} b_j(\tau,x)\\
\label{eq:C8}&\;\;\;\;\;\;\;\cdot\frac{\partial U_i(t_0,\tau)U_j(t_0,\tau)}{\partial \tau}d\tau
\end{align}
and by using $\frac{\partial U_i(t_0,\tau)U_j(t_0,\tau)}{\partial \tau} = u_i(\tau,\omega\tau)U_j(t_0,\tau)+ u_j(\tau,\omega\tau) U_i(t_0,\tau)$ for almost all $\tau$, $i=1,\ldots,m$, $j=1,\ldots,m$. Note that, $R_3$ and $R_4$ contain the rest terms after relabeling the indices. Furthermore, $R_3$ contains the terms where $i=j$, which is treated as a special case. 

We now turn to \eqref{eq:liebracket_system}. By assumption, the solution $z:\mathbb{R}\to\mathbb{R}^n$ of \eqref{eq:liebracket_system} exists and $z(t)$ is bounded for $t=[t_0,\infty)$ and for all $z(t_0)=z_0\in \mathcal{B}$. Thus, $z(t)$ that can be written as
\begin{equation}
\begin{split}
\!\!\!\!\!z(t) = z_0\! +\!\! \int_{t_0}^{t}\!\!b_0(\tau,z)+\! \sum_{\substack{i=1\\j=i+1}}^m[b_i,b_j](\tau,z)\nu_{ji}(\tau)d\tau\!\!\!\!
\end{split}
\end{equation}
with $t\in[t_0,\infty)$, $z(t_0)=z_0$ and $\nu_{ji}(t)$ as defined in \eqref{eq:liebracket_system}.

In the following, we show that the distance between $x(t)$ and $z(t)$ with $z(t_0)=x(t_0)=x_0$ can be made arbitrary small on a finite time interval with $\omega$ chosen sufficiently large. Choose $z(t_0)=x(t_0)=x_0\in \mathcal{K}$ and since $\mathcal{K}\subseteq\mathcal{B}$ is bounded and since solutions initialized in $\mathcal{B}$ stay uniformly bounded, there exists a bounded set $\mathcal{M}\subseteq\mathbb{R}^n$ such that for all $t_0\in\mathbb{R}$ and all $z(t_0)\in\mathcal{K}$ we have $z(t) \in \mathcal{M}$, $t\in[t_0,\infty)$. 
Define a tubular set around $z(t)$, i.e. $\mathcal{O}(t) =\{x\in\mathbb{R}^n: |x-z(t)|\leq D\}$, $t\in[t_0,\infty)$. 
We now consider the case, where we assume that there exists a time $t_D(t_0,x_0,\omega)$ with $0< t_D(t_0,x_0,\omega)<\bar{t}_e$ such that $x(t)=x(t;t_0,x_0,\omega)$ leaves $\mathcal{O}(t)$ at $t_0+t_D(t_0,x_0,\omega)$ and with $\bar{t}_e$ the maximal time of existence of $x(t)$. The trivial case is given, when $x(t)\in\mathcal{O}(t)$ for all $t\in[t_0,\infty)$. 

Let $t_f\in(0,\infty)$ be given. We now show that there exists an $\omega_0\in(0,\infty)$ such that for every $\omega\in[\omega_0,\infty)$, every $t_0\in\mathbb{R}$ and every $x_0\in\mathcal{K}$ we have $t_D:=t_D(t_0,x_0,\omega)\geq t_f$. Suppose for the sake of contradiction that there exists a $t_0\in\mathbb{R}$ and an $x_0\in\mathcal{K}$ such that for all $\omega_0\in(0,\infty)$ we have that there exists an $\omega\in (\omega_0,\infty)$ such that $t_D(t_0,x_0,\omega)< t_f$.  

\begin{figure}[htbp]
\centering
\includegraphics[width=180pt]{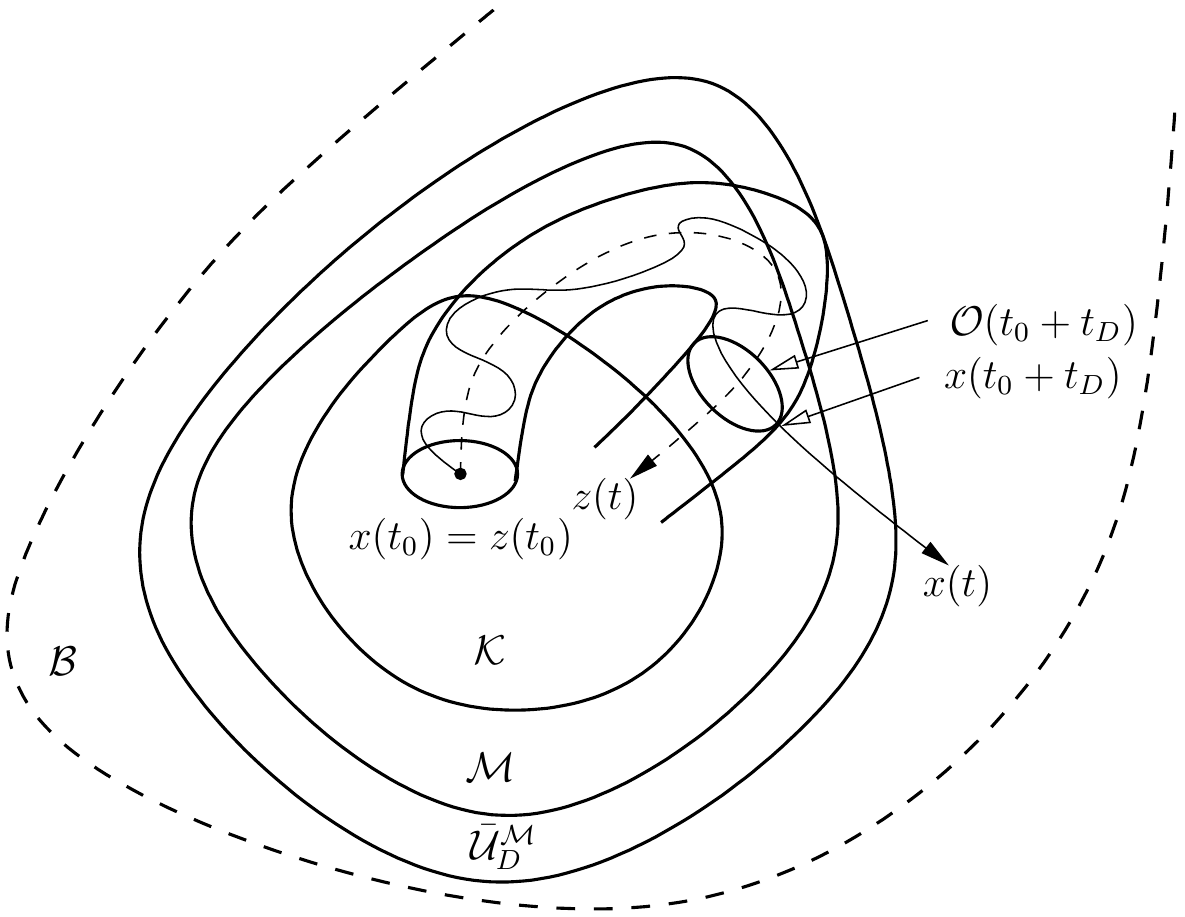}
\caption{$x(t)$ stays in $\bar{\mathcal{U}}^\mathcal{M}_D$ for all $t\in[t_0,t_0+t_D]$}
\label{fig:proofsets}
\end{figure}


Consider the distance between $x(t)$ and $z(t)$ through $z(t_0)=x(t_0)$ for $t\in[t_0,t_0+t_D]$. We add and subtract the expression $\int_{t_0}^{t}[b_i,b_j](\tau,x)\nu_{ji}(\tau)d\tau$ and obtain
\begin{align} 
\nonumber x(t)-z(t)=& \int_{t_0}^{t}b_0(\tau,x)-b_0(\tau,z)\\
\nonumber&+\!\!\!\!\sum_{\substack{i=1\\j=i+1}}^m\!\!\!\biggl( [b_i,b_j](\tau,x) - [b_i,b_j](\tau,z)\biggr)\nu_{ji}(\tau)d\tau\\
\label{eq:diff}&+ R_1+R_2+R_3+R_4+R_5
\end{align}
with
\begin{equation}
\begin{split}
R_5 :=\sum_{i=1}^m\sum_{j=i+1}^m\int_{t_0}^{t} [b_i,b_j](\tau,x) V_{ji}(\tau,\omega\tau)d\tau.
\end{split}
\end{equation}
and $V_{ji}(\tau,\omega\tau)= \omega u_j(\tau,\omega\tau)U_i(t_0,\tau)-\nu_{ji}(\tau)$.

Suppose for the moment that there exist $k\in[0,\infty)$ and $\omega^*_0\in(0,\infty)$ such that for every $\omega\in(\omega^*_0,\infty)$, every $t_0\in\mathbb{R}$ and every $x_0\in\mathcal{K}$ we have $\sum_{i=1}^5|R_i|\leq \frac{k}{\sqrt{\omega}}$, $t\in[t_0,t_0+t_D]$. Note that $x(t),z(t)\in \bar{\mathcal{U}}^\mathcal{M}_D$, $t\in[t_0,t_0+t_D]$ (see Fig. \ref{fig:proofsets}) and note that with Assumption A\ref{ass:a3} we have that $|\nu_{ji}(\tau)|\leq\frac{1}{T}\int_0^T|u_j(\tau,\theta)\int_0^\theta u_i(\tau,s)ds d\theta|\leq \frac{1}{2}M_jM_iT$. Thus, with Assumption A\ref{ass:a1} we have for the compact set $\bar{\mathcal{U}}^\mathcal{M}_D$ that there exists an $L\in(0,\infty)$ such that $|b_0(\tau,x)-b_0(\tau,z)+\sum_{\substack{i=1\\j=i+1}}^m( [b_i,b_j](\tau,x) - [b_i,b_j](\tau,z))\nu_{ji}(\tau)|\leq L|x(\tau)-z(\tau)|$ and therefore
\begin{equation}
\begin{split}
|x(t)-z(t)| \leq& \int_{t_0}^{t}L|x(\tau)-z(\tau)|d\tau+ \frac{k}{\sqrt{\omega}}
\end{split}
\end{equation}
with $x(t),z(t)\in\bar{\mathcal{U}}^\mathcal{M}_D$, $t\in[t_0,t_0+t_D]$ and $\omega\in(\omega_0^*,\infty)$. Using the Lemma of Gronwall-Bellman we obtain
\begin{equation}
\label{eq:abstand}
\begin{split}
|x(t)-z(t)| \leq& \frac{k}{\sqrt{\omega}}e^{L(t-t_0)}, t\in[t_0,t_0+t_D].
\end{split}
\end{equation}
Choose now  $\omega_0 = \max\{ \frac{4k^2e^{2Lt_f}}{D^2}, \omega^*_0\}$, which is \emph{independent} of $t_0\in\mathbb{R}$ and $x_0\in\mathcal{K}$. 
Now suppose that $t_D<t_f$, but since for \emph{every} $\omega\in(\omega_0,\infty)$, \emph{every} $t_0\in\mathbb{R}$ and \emph{every} $x_0\in\mathcal{K}$ we have with \eqref{eq:abstand} that $|x(t)-z(t)|<D$, $t\in[t_0,t_0+{t}_D]$, thus ${t}_D$ can not be the time, when $x(t)$ leaves $\mathcal{O}(t)$ which contradicts $t_D<t_f$. 
Furthermore, since $\omega_0$ is independent of $t_0,x_0$, the estimate holds for \emph{every} $t_0\in\mathbb{R}$ and \emph{every} $x_0\in\mathcal{K}$. Thus, we conclude that for every bounded set $\mathcal{K}\subseteq \mathcal{B}$, for every $D\in(0,\infty)$ and every $t_f\in(0,\infty)$ there exists an $\omega_0\in(0,\infty)$ such that for every $\omega\in(\omega_0,\infty)$, for every $t_0\in\mathbb{R}$ and for every $x_0\in\mathcal{K}$ there exist solutions $x$ and $z$ through $x(t_0)=z(t_0)=x_0$ which satisfy $|x(t)-z(t)|< D$, $t\in[t_0,t_0+t_f]$.

It remains to show that there exist $k\in[0,\infty)$ and $\omega^*_0\in(0,\infty)$ such that for every $\omega\in(\omega^*_0,\infty)$, every $t_0\in\mathbb{R}$ and every $x_0\in\mathcal{K}$ we have $\sum_{i=1}^5|R_i|\leq \frac{k}{\sqrt{\omega}}$, $t\in[t_0,t_0+t_D]$. 
Following the same lines as in \cite{Moreau:2003fk} the expressions $|R_i|, i=1,\ldots,5$ decay uniformly to zero with $\omega\to\infty$ on compact sets. 
Due to space limitations, this is shown only for $R_5$. The procedure is similar for $R_1$ to $R_4$. 

Note that for every $x_0\in\mathcal{K}$ we have that $x(t)\in\bar{\mathcal{U}}^\mathcal{M}_D$, $t\in[t_0,t_0+t_D]$. Due to Assumption A\ref{ass:a1}, the vector fields $b_i$, $i=1,\ldots,m$ are twice continuously differentiable and thus we can perform a partial integration which yields for $R_5$
\begin{align} 
\nonumber R_5=&\sum_{\substack{i=1\\j=i+1}}^m[b_i,b_j](t,x) \int_{t_0}^{t}V_{ji}( \tau,\omega\tau)d\tau\\
\nonumber&- \int_{t_0}^{t}\biggl[\biggl(\frac{\partial[b_i,b_j](\tau,x)}{\partial x}\dot{x}\\
\label{lemeq:11}
&\;\;\;\;\;\;+\!\frac{\partial[b_i,b_j](\tau,x)}{\partial \tau}\biggr)\!\!\int_{t_0}^\tau\!\!\! V_{ji}(\theta,\omega\theta)d\theta\biggr] d\tau.
\end{align}
Substituting $\dot{x}(\tau)=b_0(\tau,x(\tau))+\sum_{i=1}^{m}$ $b_{i}(\tau,x(\tau))\sqrt{\omega}$ $u_{i}(\tau,\omega \tau)$ yields
\begin{align}
\nonumber &R_5=\sum_{\substack{i=1\\j=i+1}}^m[b_i,b_j](t,x) \int_{t_0}^{t}V_{ji}( \tau,\omega\tau)d\tau\\
\nonumber&\!-\!\!\!\int_{t_0}^{t}\biggl[\biggl(\frac{\partial[b_i,b_j](\tau,x)}{\partial x}\biggl(b_0(\tau,x)+\!\!\sum_{i=1}^{m}b_{i}(\tau,x)\sqrt{\omega}u_{i}(\tau,\omega \tau)\!\biggr)\\
&\;\;\;+\frac{\partial[b_i,b_j](\tau,x)}{\partial \tau}\biggr)\int_{t_0}^\tau V_{ji}(\theta,\omega\theta)d\theta\biggr] d\tau.
\end{align}

Due to Assumptions A\ref{ass:a1}, A\ref{ass:a2} and A\ref{ass:a3} there exist for $\bar{\mathcal{U}}^\mathcal{M}_D$ constants $C_1, \ldots, C_4 \in[0,\infty)$ such that $|[b_i,b_j](t,x)|\leq C_1$, $|\frac{\partial[b_i,b_j](\tau,x)}{\partial x}|\leq C_2$, $|b_0(\tau,x)+\frac{\partial[b_i,b_j](\tau,x)}{\partial \tau}|\leq C_3$ and $|\sum_{i=1}^{m}b_{i}(\tau,x) u_{i}(\tau,\omega \tau)|\leq C_4$ for every $t, \tau \in\mathbb{R}$ and every $x\in\bar{\mathcal{U}}^\mathcal{M}_D$. This yields 
\begin{align}
\nonumber |R_5|\leq&\;\sum_{\substack{i=1\\j=i+1}}^m C_1\biggl|\int_{t_0}^{t}V_{ji}(\tau,\omega\tau)d\tau\biggr|\\
&+\!\!\int_{t_0}^{t}\!\!\!C_2(C_3+\sqrt{\omega}C_4)\biggl|\int_{t_0}^\tau\!\!\! V_{ji}(\theta,\omega\theta)d\theta \biggr|d\tau.
\end{align}
Furthermore, the functions $u_i$, $i=1,\ldots,m$ satisfy the assumptions of Lemma \ref{lem:average3} and thus, there exist $k^{ji}_1,k^{ji}_2,k^{ji}_3,k^{ji}_4,k^{ji}_5\in[0,\infty)$ such that $|\int_{t_0}^t V_{ji}(\tau,\omega\tau)d\tau|\leq k^{ji}_1\frac{(t-t_0)^2}{\omega}+k^{ji}_2\frac{t-t_0}{\omega} +k^{ji}_3\frac{1}{\omega}+k^{ji}_4\frac{1}{\omega^2}+k^{ji}_5\frac{1}{\omega^3}$ and also $\int_{t_0}^{t}|\int_{t_0}^\tau V_{ji}(\theta,\omega\theta)d\theta|d\tau\leq k^{ji}_1\frac{(t-t_0)^3}{3\omega}+k^{ji}_2\frac{(t-t_0)^2}{2\omega} +k^{ji}_3\frac{(t-t_0)}{\omega}+k^{ji}_4\frac{(t-t_0)}{\omega^2}+k^{ji}_5\frac{(t-t_0)}{\omega^3}$. 
From these estimates it becomes clear that there exist $k_{0,5}\in[0,\infty)$ and $\omega_{0,5}\in(0,\infty)$ such that for every $\omega\in(\omega_{0,5},\infty)$, $t_0\in\mathbb{R}$ and every $x_0\in\mathcal{K}$ we have $|R_5|\leq \frac{k_{0,5}}{\sqrt{\omega}}$, $t\in[t_0,t_0+t_D]$.

Estimates for $R_1$ and $R_2$ follow immediately from Assumptions A\ref{ass:a1} to A\ref{ass:a4} and Lemma \ref{lem:average1}. For the expressions $R_3$ and $R_4$ a partial integration and Lemma \ref{lem:average1} yields a similar result. Thus, there exist $k_{0,i},\omega_0^i$, such that $|R_i|\leq\frac{k_{0,i}}{\sqrt{\omega}}$, $\omega\in(\omega_{0,i},\infty)$, $i=1,\ldots,5$ respectively. 
Summarizing, there exist $k=5\max_i\{k_{0,i}\}$ and $\omega^*_0=\max_i\{\omega_{0,i}\}$ such that for all $\omega\in(\omega^*_0,\infty)$, every $t_0\in\mathbb{R}$ and every $x_0\in\mathcal{K}$ we have 
\begin{equation}
\label{eq:upperboundRi}
\sum_{i=1}^5|R_i|\leq \frac{k}{\sqrt{\omega}}, t\in[t_0,t_0+t_D].
\end{equation}

\section{Proof of Theorem \ref{thm:liesystem_loc}}
\label{pf:liesystem_loc}
The proof follows the same argumentation as in \cite{871771} but
extends it to the stability of a compact set.

\textbf{Practical uniform stability} We show now that $\mathcal{S}$ is practically uniformly stable for \eqref{eq:input_affine_system}, see Definition 1. First, since the set $\mathcal{S}$ is locally uniformly asymptotically stable for \eqref{eq:liebracket_system} there exists a $\delta_1\in(0,\infty)$ such that $\mathcal{S}$ is $\delta_1$-uniformly attractive for \eqref{eq:liebracket_system}. Take an arbitrary $\epsilon\in(0,\infty)$ and let $B_{1} \in (0,\epsilon)$.  Since $\mathcal{S}$ is uniformly stable for \eqref{eq:liebracket_system}, there exists a $\delta \in (0,\delta_1)$ 
such that for all $t_{0} \in \mathbb{R}$
\begin{equation}
\label{eq:1}
\begin{split}
z(t_0) \in \mathcal{U}_{\delta}^\mathcal{S} \Rightarrow z(t)\in \mathcal{U}_{B_1}^\mathcal{S}, \; t \in [t_{0},\infty). \\
\end{split}
\end{equation}
Second observe that, since the set $\mathcal{S}$ is $\delta_1$-uniformly attractive for \eqref{eq:liebracket_system} 
and $\delta\in(0,\delta_1)$ we have that for every $B_{2} \in (0,\delta)$ there exists a time $t_f \in (0, \infty)$ such that for all $t_{0}\in \mathbb{R}$
\begin{equation}
\label{eq:2} 
\begin{split}
z(t_0) \in \mathcal{U}_{\delta}^\mathcal{S} \Rightarrow z(t)\in \mathcal{U}_{B_2}^\mathcal{S},\; t \in [t_{0}+t_f, \infty).
\end{split}
\end{equation}
Let $D= \min\{\epsilon-B_{1},\delta-B_{2}
\}$, $\mathcal{B}=\mathcal{K}={\mathcal{U}}_\delta^\mathcal{S}$ and $t_f$ determined above. Because of \eqref{eq:1} the set $\mathcal{B}$ satisfies \eqref{eq:defB}. Due to Theorem 1, there exists an $\omega_{0}\in(0,\infty)$ such that for all $t_0\in\mathbb{R}$ and for all $\omega\in(\omega_{0},\infty)$ and all $x(t_0)\in \mathcal{K}$ we have that $|x(t)-z(t)|<D$, $t\in[t_0,t_0+t_f]$. 
This together with \eqref{eq:1} and \eqref{eq:2} yields for all $\omega\in (\omega_{0},\infty)$
\begin{equation}
\label{eq:4}
\begin{split}
x(t_0) &\in \mathcal{U}_\delta^\mathcal{S} \Rightarrow
 x(t)\in \mathcal{U}^\mathcal{S}_\epsilon, \\
 &t \in [t_{0}, t_{0}+t_f] \text{ and } x(t_{0}+t_f) \in \mathcal{U}^S_\delta.
\end{split}
\end{equation}
Since $x(t_{0}+t_f) \in \mathcal{U}^\mathcal{S}_\delta$ a repeated application of the procedure with another solution $z(t)$ of \eqref{eq:liebracket_system}  through $x(t_0+t_f)$ and the same choice of $D$, $\mathcal{K}$ and $t_f$ as above yields for all $t_{0} \in \mathbb{R}$ and for all $\omega\in (\omega_{0},\infty)$
\begin{equation}
\begin{split}
x(t_0) \in \mathcal{U}_\delta^\mathcal{S} \Rightarrow x(t)\in \mathcal{U}^\mathcal{S}_\epsilon, t \in [t_{0}, \infty).
\end{split}
\end{equation}

\textbf{Practical uniform attractivity}
We show now that there exists a $\delta\in(0,\infty)$ such that $\mathcal{S}$ is $\delta$-practically uniformly attractive for \eqref{eq:input_affine_system}, see Definition 2. Since the set $\mathcal{S}$ is locally uniformly asymptotically stable for \eqref{eq:liebracket_system} there exists a $\delta_1\in(0,\infty)$ such that $\mathcal{S}$ is $\delta_1$-uniformly attractive for \eqref{eq:liebracket_system}. Furthermore, by uniform stability there exists a ${\delta}_2\in(0,\infty)$ such that for all $t_0\in\mathbb{R}$ we have that
\begin{equation}
\label{eq:new1}
\begin{split}
z(t_0) \in \mathcal{U}_{{\delta}_2}^\mathcal{S} \Rightarrow z(t)\in \mathcal{U}_{\delta_1}^\mathcal{S}, \; t \in [t_{0},\infty).
\end{split}
\end{equation}
Choose some $\epsilon \in (0, \infty)$. 
By practical uniform stability proven above, there exist $B_3 \in (0, \infty)$ and $\omega_{0,1} \in (0,\infty)$ such that for all $t_{0} \in \mathbb{R}$ and for all $\omega \in (\omega_{0,1},\infty)$
\begin{equation}
\label{eq:6}
\begin{split}
x(t_0) \in \mathcal{U}^\mathcal{S}_{B_3} \Rightarrow x(t)\in \mathcal{U}^\mathcal{S}_\epsilon,  t \in [t_{0}, \infty).
\end{split}
\end{equation}
Let $B_4 \in (0, B_3)$ and $\delta \in (0,{\delta}_2)$. Note that $\delta<\delta_2\leq\delta_1$. Since the set $\mathcal{S}$ is $\delta_1$-uniformly attractive for \eqref{eq:liebracket_system}, there exists a $t_f \in (0, \infty)$ such that for all $t_{0} \in \mathbb{R}$
\begin{equation}
\begin{split}
\label{eq:7}
z(t_0) \in \mathcal{U}^\mathcal{S}_{\delta} \Rightarrow z(t)\in \mathcal{U}^\mathcal{S}_{B_4}, t \in [t_{0}+t_f,\infty).
\end{split}
\end{equation}
Let $D = B_3-B_{4}$, $\mathcal{B}=\mathcal{K}={\mathcal{U}}_\delta^\mathcal{S}$ and $t_f$ determined above. Because of \eqref{eq:new1} the set $\mathcal{B}$ satisfies \eqref{eq:defB}. Due to Theorem 1, there exists an $\omega_{0,2} \in  (0,\infty)$ such that for all $t_{0} \in \mathbb{R}$ and for all $\omega \in (\omega_{0,2},\infty)$ and all $x(t_0)\in \mathcal{K}$ we have that $|x(t) - z(t)| < D, \; t \in [t_{0},t_{0}+t_f]$. 
This estimate together with \eqref{eq:7} yield for all $t_{0} \in \mathbb{R}$ and for all $\omega \in (\omega_{0,2},\infty)$
\begin{equation}
\begin{split}
x(t_0) \in{\mathcal{U}}_\delta^\mathcal{S} \Rightarrow x(t_{0} + t_f) \in \mathcal{U}^\mathcal{S}_{B_3}.
\end{split}
\end{equation}
With \eqref{eq:6}, this leads for all $t_{0} \in \mathbb{R}$ and for all $\omega \in (\omega_{0,2},\infty)$ where $\omega_0 = \max\{\omega_{0,1},\omega_{0,2}\}$ to
\begin{equation}
\begin{split}
x(t_0)\in \mathcal{U}^\mathcal{S}_{\delta} \Rightarrow x(t)\in \mathcal{U}^\mathcal{S}_\epsilon, t \in [t_{0} + t_f, \infty).
\end{split}
\end{equation}
This is the last property we had to prove.

\end{appendix}

\bibliography{bibliography}

\end{document}